\documentclass[12pt, reqno]{amsart}
\usepackage{fullpage}
\usepackage{amsmath, amsthm, amssymb}
\usepackage{mathrsfs}
\usepackage{enumerate}
\usepackage{tikz}
\usetikzlibrary{calc}
\tikzstyle{v}=[circle, draw, fill=white,
                        inner sep=0pt, minimum width=4pt]
\tikzset{>=stealth}
\newcommand{\dr}[1]{\rule[-3ex]{0pt}{0pt}\rule[4ex]{0pt}{0pt} \raisebox{\dimexpr-.5\height+.5\ht\strutbox\relax}{\tikz\draw#1;}}
\newcommand{\vc}[1]{\rule[-3ex]{0pt}{0pt}\rule[4ex]{0pt}{0pt} \raisebox{\dimexpr-.5\height+.5\ht\strutbox\relax}{#1}}

\usepackage{longtable}
\usepackage{array}

\newtheorem{thm}{Theorem}[section]
\newtheorem{proposition}[thm]{Proposition}
\newtheorem{prop}[thm]{Proposition}
\newtheorem{cor}[thm]{Corollary}
\newtheorem{lemma}[thm]{Lemma}
\newtheorem{conj}[thm]{Conjecture}
\newtheorem{alg}[thm]{Algorithm}

\theoremstyle{definition}
\newtheorem*{definition}{Definition}
\newtheorem*{defn}{Definition}
\newtheorem{ex}[thm]{Example}
\newtheorem{example}[thm]{Example}
\newtheorem{question}[thm]{Question}

\theoremstyle{remark}
\newtheorem{rmk}[thm]{Remark}

\newcommand{\FK}{\mathcal E}
\newcommand{\RR}{\mathbb R}
\newcommand{\QQ}{\mathbb Q}
\newcommand{\ZZ}{\mathbb Z}

\newcommand{\Hilb}{\mathcal H}

\renewcommand{\H}{\ensuremath{\mathscr{H}}}
\newcommand{\eW}{\ensuremath{\widehat \S_n}}

\newcommand{\eH}{\ensuremath{\widehat{\H}}} 
\newcommand{\pH}{\widehat{\H}^+} 

\newcommand{\aW}{\ensuremath{\widetilde{\S}_n}}
\newcommand{\aWnil}{\ensuremath{\widetilde{\Nil}}}
\newcommand{\fWnil}{\ensuremath{\Nil}}
\newcommand{\eWnil}{\ensuremath{\widehat{\Nil}}}

\newcommand{\e}{\mathsf}
\newcommand{\idelm}{\mathrm{id}}
\renewcommand{\S}{\mathfrak S}
\newcommand{\rev}{\operatorname{rev}}
\newcommand{\supp}{\Pi}
\newcommand{\Nil}{\mathcal N}
\newcommand{\act}{*}
\newcommand{\factL}{!_{\scriptscriptstyle\searrow}}
\newcommand{\factR}{!_{\scriptscriptstyle\swarrow}}
\newcommand{\leftexp}[2]{{\vphantom{#2}}^{#1}{#2}}

\newcommand{\affineA}[1]{{\ensuremath{\tilde{A}_{#1}}}}
\newcommand{\cyc}{{\affineA{n-1}}}
\newcommand{\C}{\ensuremath{C^{\prime}}}
\renewcommand{\subset}{\subseteq}

\newlength{\cellsize}
\newcommand\tableau[1]{
\vcenter{
\let\\=\cr
\baselineskip=-16000pt
\lineskiplimit=16000pt
\lineskip=0pt
\halign{&\tableaucell{##}\cr#1\crcr}}}

\newcommand{\tableaucell}[1]{{%
\def \arg{#1}\def \void{}%
\ifx \void \arg
\vbox to \cellsize{\vfil \hrule width \cellsize height 0pt}%
\else
\unitlength=\cellsize
\begin{picture}(1,1)
\put(0,0){\makebox(1,1){$#1$}}
\put(0,0){\line(1,0){1}}
\put(0,1){\line(1,0){1}}
\put(0,0){\line(0,1){1}}
\put(1,0){\line(0,1){1}}
\end{picture}%
\fi}}

\begin{document}
\title{Subalgebras of the Fomin-Kirillov algebra}
\author{Jonah Blasiak}
\address{Department of Mathematics, Drexel University, Philadelphia, PA 19104}
\email{jblasiak@gmail.com}
\thanks{J. Blasiak was partially supported by NSF Grant DMS 1161280.}

\author{Ricky Ini Liu}
\address{Department of Mathematics, University of Michigan, Ann Arbor, MI 48109}
\email{riliu@umich.edu}
\thanks{R. I. Liu was partially supported by NSF Postdoctoral Research Fellowship DMS 1004375.}

\author{Karola M\'esz\'aros}
\address{Department of Mathematics, Cornell University, Ithaca, NY 14853}
\email{karola@math.cornell.edu}
\thanks{K. M\'esz\'aros was partially supported by NSF Postdoctoral Research Fellowship DMS 1103933.}

\begin{abstract}
The Fomin-Kirillov algebra $\FK_n$ is a noncommutative quadratic algebra with a generator for every edge of the complete graph on $n$ vertices. For any graph $G$ on $n$ vertices, we define ${\FK_G}$ to be the subalgebra of $\FK_n$ generated by the edges of $G$. We show that these algebras have many parallels with Coxeter groups and their nil-Coxeter algebras: for instance, $\FK_G$ is a free $\FK_H$-module for any $H\subseteq G$, and if $\FK_G$ is finite-dimensional, then its Hilbert series has symmetric coefficients. We determine explicit monomial bases and Hilbert series for $\FK_G$ when $G$ is a simply-laced finite Dynkin diagram or a cycle, in particular showing that $\FK_G$ is finite-dimensional in these cases.
We also present conjectures for the Hilbert series of $\FK_{\tilde{D}_n}$,  $\FK_{\tilde{E}_6}$, and $\FK_{\tilde{E}_7}$,
as well as for which graphs $G$ on six vertices $\FK_G$ is finite-dimensional.
\end{abstract}

\maketitle

\section{Introduction}

The Fomin-Kirillov algebra $\FK_n$  \cite{FK} is a certain noncommutative algebra with generators $x_{ij}$ for $1 \leq i < j \leq n$ that satisfy a simple set of quadratic relations. 
While it was originally introduced as a tool to study the structure constants for Schubert polynomials, since then the Fomin-Kirillov algebra and its generalizations have received much attention from the perspectives of both combinatorics and algebra: see, for instance, \cite{Bazlov, FP, Kirillov, KirillovMaeno,  Lenart, LenartMaeno, Majid, MPP,  MS, P, Vendramin}. But despite its simple presentation, even some basic questions about $\FK_n$ have eluded an answer thus far, such as whether or not it is finite-dimensional for $n \geq 6$.

In order to better understand the structure of $\FK_n$, we consider the following subalgebras: 
\[
\parbox{14cm}{for any graph $G$ on vertices $1, 2, \dots, n$, the \emph{Fomin-Kirillov algebra ${\FK_G}$ of $G$} is the subalgebra of $\FK_n$ generated by $x_{ij}$ for all edges $\overline{ij}$ in $G$.}
\]
While this definition might seem hopelessly ingenuous, our initial computations using the algebra package \texttt{bergman} \cite{bergman} revealed that, remarkably, whenever $G$ is a graph with at most five vertices, $\FK_G$ has a one-dimensional top degree component, and in fact its Hilbert series has symmetric coefficients. (See the Appendix for these computations.) Our current study is therefore dedicated to an investigation of these beautiful, yet mysterious algebras.

\medskip

The first half of this paper is devoted to proving structural properties of Fomin-Kirillov algebras. We prove that any finite-dimensional $\FK_G$ has a Hilbert series with symmetric coefficients, as well as that $\FK_G$ is a free $\FK_H$-module whenever $H$ is a subgraph of $G$. We also demonstrate that the Fomin-Kirillov algebras exhibit a striking amount of structure, much of which parallels Coxeter groups and nil-Coxeter algebras: for instance, we describe analogues of minimal coset representatives, descent sets, Bruhat order, and long words.

The key tools to proving these facts can be derived from a braided Hopf algebra structure on $\FK_n$ \cite{FP, MS}---in this context, the subalgebras $\FK_G$ are \emph{(left) coideal subalgebras}. Coideal subalgebras are important objects in the study of Hopf algebras, and they sometimes possess freeness and symmetry properties analogous to the ones that we exhibit for $\FK_G$ \cite{Masuoka}. There also exists a certain symmetric bilinear form on $\FK_n$ whose nondegeneracy (known for $n \leq 5$ \cite{Grana}) implies that $\FK_n$ is a special type of braided Hopf algebra called a Nichols algebra \cite{AS, MS}. Though Nichols algebras have been studied since \cite{Nichols}, the many parallels to Coxeter groups demonstrated here have not been observed in the literature on Nichols algebras to our knowledge. In our exposition below, we will not assume familiarity with braided Hopf algebras or Nichols algebras. We refer the reader to \cite{AS} for more details about these objects.

\medskip

A particularly exciting part of this study is the abundance of graphs $G$ for which $\FK_G$ is finite-dimensional---a much larger class than finite-dimensional Coxeter groups---and the combinatorial mysteries awaiting discovery here. In the second half of this paper, we present results obtained so far about these finite-dimensional algebras. Specifically, we determine explicit monomial bases and Hilbert series for $\FK_G$ when $G$ is a simply-laced Dynkin diagram or a cycle. Here, another surprising connection between the Fomin-Kirillov algebras and Coxeter groups appears: for $G$ a simply-laced Dynkin diagram, the dimension of $\FK_G$ is equal to the dimension of the Weyl group of $G$ divided by the index of connection. In the case when $G$ is a cycle on $n$ vertices, we describe $\FK_G$ explicitly as a quotient of the nil-Coxeter algebra of the affine symmetric group, showing that it is essentially a $q=0$ version of the twisted product of the group algebra of the symmetric group and the ring of coinvariants. We also present intriguing conjectures for the Hilbert series of $\FK_{\tilde{D}_n}$,  $\FK_{\tilde{E}_6}$, and $\FK_{\tilde{E}_7}$, as well as for which graphs $G$ on six vertices $\FK_G$ is finite-dimensional.

\medskip

This paper is organized as follows. In Section 2, we introduce $\FK_n$ and briefly describe some examples of $\FK_G$. In Section 3, we discuss some structural properties of $\FK_n$ and prove that whenever $\FK_G$ is finite-dimensional, its Hilbert series has symmetric coefficients. In Section 4, we prove that when $H$ is a subgraph of $G$, $\FK_G$ is a free $\FK_H$-module. We also show that $\FK_n$ has a tensor product decomposition with factors given by certain complementary $\FK_G$. In Section 5, we discuss Coxeter groups and nil-Coxeter algebras, as well as their relationships and similarities to the subalgebras $\FK_G$. In Section 6, we describe $\FK_G$ for $G$ a simply-laced Dynkin diagram, computing its Hilbert series in each case. In Section 7, we describe $\FK_G$ when $G$ is a cycle (that is, an affine Dynkin diagram of type $\tilde{A}_{n-1}$). Finally, we close in Section 8 with some open questions and conjectures to guide further research. We also include an Appendix containing the Hilbert series of $\FK_G$ for all connected graphs $G$ on at most five vertices.

\section{Preliminaries}

We begin with the definition of the Fomin-Kirillov algebra \cite{FK}.

\begin{defn}
	The \emph{Fomin-Kirillov algebra} $\FK_n$ is the quadratic algebra (say, over $\QQ$) with generators
	$x_{ij}=-x_{ji}$ for $1 \leq i < j \leq n$ with the following relations:
	\begin{itemize}
		\item $x_{ij}^2 = 0$ for distinct $i,j$;
		\item $x_{ij}x_{kl} = x_{kl}x_{ij}$ for distinct $i,j,k,l$;
		\item $x_{ij}x_{jk}+x_{jk}x_{ki}+x_{ki}x_{ij}=0$ for distinct $i,j,k$.
	\end{itemize}
\end{defn}

Let $V$ be the vector space spanned by the generators $x_{ij}$. Then $\FK_n$ is a quotient of the tensor algebra $T(V) = \bigoplus_{n \geq 0} V^{\otimes n}$, the free associative algebra on the generators of $\FK_n$.

Since the relations are homogeneous, $\FK_n$ is graded with respect to the usual degree. We will denote the degree $d$ part by $\FK_n^d$. 

Note that $\FK_n$ has another grading with respect to the symmetric group $\S_n$: define the $\S_n$-degree of $x_{ij}$ to be $\sigma_{ij} \in \S_n$, the transposition switching $i$ and $j$, and extend this $\S_n$-degree to all monomials in $T(V)$ by multiplicativity. Since each of the relations in $\FK_n$ is homogeneous with respect to $\S_n$-degree, this gives an $\S_n$-grading on $\FK_n$. We will write $\sigma_P$ for the $\S_n$-degree of a homogeneous element $P \in \FK_n$. (We will always specify when we mean $\S_n$-degree; the use of ``degree'' unqualified will refer to the usual notion of degree.)

There is a third grading related to the support of each monomial. For a monomial $m \in \FK_n$, define $\supp(m)$ to be the coarsest set partition of $[n]$ for which $i$ and $j$ lie in the same part if $x_{ij}$ or $x_{ji}$ appears in $m$. For instance, $\supp(x_{12}x_{23}x_{45}x_{31}) = 123|45$. Then all of the relations of $\FK_n$ are homogeneous with respect to $\supp$. Note that $\supp(m_1m_2)$ is the common coarsening of $\supp(m_1)$ and $\supp(m_2)$.

The set of relations of $\FK_n$ is also symmetric with respect to the indices. In other words, for all $\sigma \in \S_n$, there exists an automorphism of $\FK_n$ given by $\sigma(x_{ij}) = x_{\sigma(i)\sigma(j)}$.

Also note that $\FK_n$ is isomorphic to its opposite algebra: the linear map $\rev\colon T(V) \to T(V)$ sending any monomial to the product of the same generators but in the reverse order preserves the set of relations and thus gives an antiautomorphism of $\FK_n$.

\begin{rmk}
The natural category for $\FK_n$ is the \emph{Yetter-Drinfeld category} over $\QQ[S_n]$, which is the braided monoidal category consisting of $S_n$-graded $S_n$-modules $M = \bigoplus_{\sigma \in S_n} M_\sigma$ satisfying $\sigma(M_\pi) \subset M_{\sigma\pi\sigma^{-1}}$.
\end{rmk}

\subsection{Hilbert series} \label{sec-hilbert}

Let $\Hilb_n(t)$ be the Hilbert series of $\FK_n$. Values of $\Hilb_n(t)$ for small values of $n$ are given as follows. (To simplify expressions, we will often write $[k] = 1+t+t^2+\cdots+t^{k-1}$.) 
\begin{align*}
\Hilb_1(t) =&\; 1\\
\Hilb_2(t) =&\; [2]\\
\Hilb_3(t) =&\; [2]^2[3]\\
\Hilb_4(t) =&\; [2]^2[3]^2[4]^2\\
\Hilb_5(t) =&\; [4]^4[5]^2[6]^4\\
\Hilb_6(t) =&\; 1+15t+125t^2+765t^3+3831t^4+16605t^5+64432t^6\\
&+228855t^7+755777t^8+2347365t^9+6916867t^{10}+\cdots
\end{align*}
A closed form for $\Hilb_n$ is not known for $n \geq 6$. It is not even known whether $\FK_n$ is finite-dimensional for $n \geq 6$.

\subsection{Subalgebras}

In order to better understand $\FK_n$, we will study its subalgebras. Such algebras are mentioned by Kirillov in the introduction of \cite{Kirillov}, and analogues for other root systems are also studied in unpublished work of Bazlov and Kirillov  \cite{Kirillov2}.
\begin{defn}
	For any graph $G$ with vertex set $[n]$, the \emph{Fomin-Kirillov algebra} ${\FK_G}$ of $G$ is the subalgebra of $\FK_n$ generated by $x_{ij}$ for all edges $\overline{ij}$ in $G$.
\end{defn}
We will write $\FK_G^d$ for the degree $d$ part of $\FK_G$ and $\FK_G^+ = \bigoplus_{d \geq 1} \FK_G^d$ for the positive degree part of $\FK_G$.

\medskip
Note that by this definition, $\FK_n=\FK_{K_n}$. If graphs $G$ and $G'$ are isomorphic, then so are the algebras $\FK_G$ and $\FK_{G'}$ since we can apply the automorphism of $\FK_n$ that permutes the indices appropriately. Moreover, since $\FK_n$ is a subalgebra of $\FK_{n+1}$ (as can be seen from considering $\supp$-degree), $\FK_G$ does not depend on the choice of $n$---that is, it does not change if we add or remove isolated vertices. Finally, if $G$ and $H$ are graphs on disjoint vertex sets, then all variables in $\FK_G$ commute with all variables in $\FK_H$, so $\FK_{G+H} \cong \FK_G \otimes \FK_H$.

\medskip
Given a subalgebra $\FK_G$, we will write $\Hilb_G(t)$ for its Hilbert series. Values of $\Hilb_G$ for all connected graphs $G$ with at most five vertices are given in the Appendix. Computations were performed using the algebra package \texttt{bergman} \cite{bergman}. Although one might not necessarily expect $\Hilb_G(t)$ to be particularly nice in general, a quick glance at the Appendix shows that in fact $\Hilb_G(t)$ is a product of cyclotomic polynomials for all $G$ with at most five vertices.

\subsection{Examples} \label{sec-examples}

Note that it is not easy to give a presentation of $\FK_G$: while $\FK_n$ is defined by quadratic relations, $\FK_G$ will usually have minimal relations that are not quadratic, possibly of much higher degree. (We may sometimes omit the word ``minimal'' when referring to minimal relations.) To illustrate this, we start with a few examples.

\begin{ex}[The Dynkin diagram $A_3$]\label{ex-a3}
	Let $G=A_3$ be the path with two edges, as shown in Figure~\ref{fig-a3}. The only quadratic relations in $\FK_{A_3}$ 
	are those that come from the definition of $\FK_3$, namely $\e{a}^2=\e{b}^2=0$.

	However, $\FK_{A_3}$ has other relations: in $\FK_3$, there is a quadratic relation involving the third edge $\e{c}$:
	\begin{equation} \label{eq1}
		0=\e{ab}+\e{bc}+\e{ca} \tag{$*$}
	\end{equation}
	Multiplying \eqref{eq1} on the right by $\e{a}$ gives
	\[0=\e{aba}+\e{bca}+\e{caa}=\e{aba}+\e{bca},\]
	while multiplying \eqref{eq1} on the left by $b$ gives
	\[0=\e{bab}+\e{bbc}+\e{bca}=\e{bab}+\e{bca}.\]
	Equating these, we deduce that $\e{aba}=\e{bab}$ in $\FK_{A_3}$, which we call a \emph{braid relation}. We will see in Theorem~\ref{thm-a} that these generate all relations, so that
	\[\FK_{A_3} = \langle \e{a},\e{b} \mid \e{a}^2=\e{b}^2=0, \;\; \e{aba}=\e{bab}\rangle.\]
	Then $\FK_{A_3}$ is the \emph{nil-Coxeter algebra} of type $A_2$. It has basis $\{\idelm, \e{a}, \e{b}, \e{ab}, \e{ba}, \e{aba}\}$ and Hilbert series
	\[\Hilb_{A_3}(t) = 1+2t+2t^2+t^3 = (1+t)(1+t+t^2) = [2][3].\]
\end{ex}

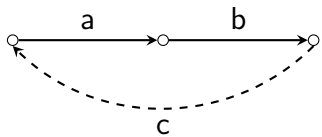
\begin{figure}
	\begin{center}
		\begin{tikzpicture}
			\node[v] (i) at (0,0) {};
			\node[v] (j) at (2,0) {};
			\node[v] (k) at (4,0) {};
			\draw[thick, ->] (i) -- node[above]{$\e{a}$} (j);
			\draw[thick, ->] (j) -- node[above]{$\e{b}$} (k);
			{\draw[dashed, thick, ->] (k.south) to [out = -135, in=-45] node[below]{$\e{c}$} (i.south);}
		\end{tikzpicture}
	\end{center}
	\caption{\label{fig-a3}
		The Dynkin diagram $A_3$ consists of the two solid edges. The label on an edge directed from vertex $i$ to vertex $j$ represents
		the generator $x_{ij}$.}
\end{figure}

\begin{ex}[The star on four vertices] \label{ex-star3}
	Consider the star graph $K_{1,3}$ on four vertices, as shown on the left of Figure~\ref{fig-star}.
	A presentation of $\FK_{K_{1,3}}$ is given by three quadratic relations:
	\[\e{a}^2=\e{b}^2=\e{c}^2=0;\]
	three braid relations:
	\[\e{aba}+\e{bab}=\e{bcb}+\e{cbc}=\e{cac}+\e{aca}=0;\]
	 and two \emph{claw relations}:
	\[\e{abca}+\e{bcab}+\e{cabc}=\e{acba}+\e{bacb}+\e{cbac}=0.\]
\end{ex}

The braid and claw relations are special cases of the following \emph{cyclic relations}.
\begin{lemma}  \label{lemma-cyclic}\cite[Lemma 7.2]{FK}
	 For $m=3, \ldots, n,$ and any distinct $a_1, \ldots, a_m \in [n]$ the following relation holds in $\FK_n$:
	\[\sum_{i=2}^m x_{a_1, a_i}x_{a_1, a_{i+1}}\cdots x_{a_1, a_m}x_{a_1, a_2}x_{a_1,a_3}\cdots x_{a_1, a_i}=0.\]
\end{lemma}
However, even star graphs have minimal relations that are not of this type.

\begin{ex}[The star on five vertices] \label{ex-star4}
	For the star graph $K_{1,4}$ on five vertices, as shown in the middle of Figure~\ref{fig-star}, $\FK_{K_{1,4}}$  has a presentation consisting of
	four quadratic relations, six braid relations, eight claw relations, six quartic cyclic relations, and three sextic relations of the following form:
	\[\e{abacdc}-\e{abcdca}+\e{acdcba}+\e{bacdcb}-\e{bcdcab}-\e{cabadc}+\e{cdabac}-\e{cdcaba}+\e{dabacd}-\e{dcabad}=0.\]
\end{ex}

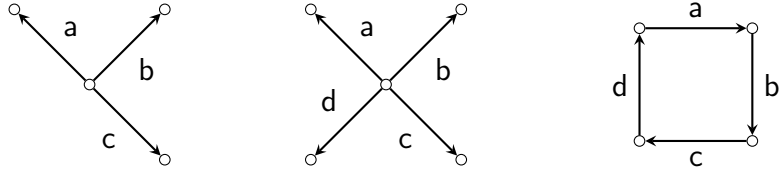
\begin{figure}
	\begin{center}
		\vc{
		\begin{tikzpicture}
			\node[v] (1) at (0,0){};
			\node[v] (2) at (-1,1){};
			\node[v] (3) at (1,1){};
			\node[v] (5) at (1, -1){};
			\draw[thick, ->] (1) -- node[above right]{$\e{a}$} (2);	
			\draw[thick, ->] (1) -- node[below right]{$\e{b}$} (3);
			\draw[thick, ->] (1) -- node[below left]{$\e{c}$} (5);		
		\end{tikzpicture}
		}
		\quad\quad\quad
		\vc{
		\begin{tikzpicture}
			\node[v] (1) at (0,0){};
			\node[v] (2) at (-1,1){};
			\node[v] (3) at (1,1){};
			\node[v] (4) at (-1,-1){};
			\node[v] (5) at (1, -1){};
			\draw[thick, ->] (1) -- node[above right]{$\e{a}$} (2);
			\draw[thick, ->] (1) -- node[below right]{$\e{b}$} (3);
			\draw[thick, ->] (1) -- node[below left]{$\e{c}$} (5);	
			\draw[thick, ->] (1) --node[above left]{$\e{d}$} (4);
		\end{tikzpicture}
		}
		\quad\quad\quad
		\vc{
		\begin{tikzpicture}
			\node[v] (1) at (0,0){};
			\node[v] (2) at (1.5,0){};
			\node[v] (3) at (1.5,-1.5){};
			\node[v] (4) at (0,-1.5){};
			\draw[thick, ->] (1) -- node[above]{$\e{a}$} (2);
			\draw[thick, ->] (2) -- node[right]{$\e{b}$} (3);
			\draw[thick, ->] (3) -- node[below]{$\e{c}$} (4);	
			\draw[thick, ->] (4) --node[left]{$\e{d}$} (1);
		\end{tikzpicture}
		}
	\end{center}
	\caption{\label{fig-star}
		On the left, the star $K_{1,3}$ on four vertices. In the middle, the star $K_{1,4}$ on five vertices. On the right, the $4$-cycle $\tilde{A}_3$.}
\end{figure}

\begin{ex}[The 4-cycle $\tilde{A}_3$] \label{ex-a3tilde}
	For the 4-cycle $\tilde{A}_3$, as shown on the right of Figure~\ref{fig-star}, $\FK_{\tilde{A}_3}$ has a presentation consisting of four quadratic relations, four braid relations, and the following three relations:
\[\e{abc}+\e{bcd}+\e{cda}+\e{dab} = 0,\]
\[\e{cba}+\e{dcb}+\e{adc}+\e{bad} = 0,\]
\[\e{abda}+\e{bcab}+\e{cdbc}+\e{dacd}+\e{acbd} + \e{bdac} = 0.\]
\end{ex}

The complexity of the relations in $\FK_G$ increases quickly as the number of edges increases. Despite this, the Fomin-Kirillov algebras often seem to be relatively well-behaved.

\section{Structure}

In this section, we will describe some structural properties of $\FK_n$ and $\FK_G$. In particular, we will show that if $\FK_G$ is finite-dimensional, then its Hilbert series has symmetric coefficients.

\subsection{A bilinear form}

The Fomin-Kirillov algebra $\FK_n$ admits an action on itself defined as follows.
\begin{prop} \label{prop-delta} \cite{FK}
There exists a unique linear map $\Delta_{ab}\colon \FK_n \to \FK_n$ satisfying
\[\Delta_{ab}(x_{ij}) = \begin{cases} 1, &\text{if $i=a$, $j=b$;}\\-1, & \text{if $i=b$, $j=a$;}\\0,&\text{otherwise;}\end{cases}\]
and $\Delta_{ab}(PQ) = \Delta_{ab}(P)\cdot Q +\sigma_{ab}(P)\cdot \Delta_{ab}(Q)$. The operators $\Delta_{ab}$ satisfy the relations of $\FK_n$, so they describe an action of $\FK_n$ on itself.
\end{prop}
We will write $\Delta_P$ for the operator corresponding to an element $P \in \FK_n$. In other words, if $P = x_{i_1j_1}x_{i_2j_2}\cdots x_{i_kj_k}$, then we let $\Delta_P = \Delta_{i_1j_1}\Delta_{i_2j_2}\cdots \Delta_{i_kj_k}$, and then we extend to all of $\FK_n$ by linearity. Note that the operators $\Delta_P$ intertwine the automorphisms $\sigma \in \S_n$ in the following way: $\sigma(\Delta_P(Q))=\Delta_{\sigma(P)}(\sigma(Q))$.

We can think of $\Delta_{ab}$ as having degree $-1$ and $\S_n$-degree $\sigma_{ab}$: if $P$ is homogeneous with respect to both degree and $\S_n$-degree, then $\Delta_{ab}P$ has degree $\deg P-1$ and $\S_n$-degree $\sigma_{ab}\sigma_P$.

Similarly, there exists a dual (right) action of $\FK_n$ on itself, defined as follows. The proof is essentially the same as that of Proposition~\ref{prop-delta} (and can also be deduced from Proposition~\ref{prop-pairing}).
\begin{prop} \label{prop-nabla}
There exists a unique linear map $\nabla_{ab}\colon \FK_n \to \FK_n$ (acting on the right) satisfying $(x_{ij})\nabla_{ab} =\Delta_{ab}(x_{ij}) $ and $(PQ)\nabla_{ab} = P\cdot (Q)\nabla_{ab} + (P)(\sigma_Q\nabla_{ab})\cdot Q$, where $\sigma_Q\nabla_{ab} = \nabla_{\sigma_Q(a)\sigma_Q(b)}$. The operators $\nabla_{ab}$ satisfy the relations of $\FK_n$, so they describe an action of $\FK_n$ on itself.
\end{prop}
We will similarly write $\nabla_P$ for the operator corresponding to an element $P \in \FK_n$.

The following lemma will be useful for performing calculations involving $\Delta_{ab}$ and $\nabla_{ab}$. It follows easily from repeated use of the Leibniz rules given in Propositions~\ref{prop-delta} and \ref{prop-nabla}.
\begin{lemma} \label{lemma-leibniz}
For a monomial $P= p_1 \cdots p_d \in \FK_n$, write $P = P_k^L p_k P_k^R$. Then
\begin{align*}
\Delta_{ab}(P) &= \sum_{k=1}^d \Delta_{ab}(p_k) \cdot  \sigma_{ab}(P_k^L) P_k^R, \text{ and}\\
(P)\nabla_{ab} &= \sum_{k=1}^d (p_k)(\sigma_{P_k^R}\nabla_{ab}) \cdot P_k^LP_k^R.
\end{align*}
\end{lemma}

\begin{ex} \label{ex-delta}
Here is a brief example of how to apply the $\Delta_{ij}$ and $\nabla_{ij}$ operators.
\begin{align*}
\Delta_{12}(x_{12}x_{23}x_{31}) &= \Delta_{12}(x_{12})\cdot x_{23}x_{31} + x_{21}\cdot \Delta_{12}(x_{23})\cdot x_{31} + x_{21}x_{13}\cdot \Delta_{12}(x_{31})\\ &= x_{23}x_{31},\\
(x_{12}x_{23}x_{31})\nabla_{12} &= x_{12}x_{23} \cdot (x_{31})\nabla_{12} + x_{12} \cdot (x_{23})\nabla_{32} \cdot x_{31} +  (x_{12})\nabla_{23} \cdot x_{23}x_{31}\\ &= -x_{12}x_{31}.
\end{align*}
\end{ex}

These actions are dual in the following sense.
\begin{prop} \label{prop-pairing}
If $P$ and $Q$ are homogeneous of the same degree,  then $\Delta_P(Q) = \Delta_Q(P) = (P)\nabla_Q = (Q)\nabla_P$. This defines a symmetric bilinear form $\langle P, Q \rangle$ on $\FK_n$. With respect to this form, the operators  $\Delta_P$ and  $\nabla_P$ are adjoint to right and left multiplication by $P$, respectively.
\end{prop}
\begin{proof}
We induct on the degree $d$ of $P$ and $Q$. Choose monomials $P=p_1\cdots p_d$ and $Q=q_1\cdots q_d$, and write $P' = P_d^L= p_1\cdots p_{d-1}$ and $Q=Q_k^Lq_kQ_k^R$ as in Lemma~\ref{lemma-leibniz}. Then
\[\Delta_P(Q) = \Delta_{P'} \Delta_{p_d}(Q)
= \sum_{k=1}^d \Delta_{P'} (\sigma_{p_d}(Q_k^L)\cdot Q_k^R) \cdot \Delta_{p_d}(q_k).\]

By induction, this equals 
\begin{equation} \label{eq-pairing}
\sum_{k=1}^d\Delta_{ \sigma_{p_d}Q_k^L}\Delta_{Q_k^R}(P')\cdot \Delta_{q_k}(p_d)
=\sum_{k=1}^d\Delta_{Q_k^L}(\sigma_{p_d}(\Delta_{Q_k^R}(P'))) \cdot \Delta_{q_k}(p_d). \tag{$*$}
\end{equation}
The $k$th term in the sum is only nonzero if $p_d = \pm q_k$, that is, if $\sigma_{p_d} = \sigma_{q_k}$. We therefore find that $\Delta_P(Q)$ equals
\[\sum_{k=1}^d\Delta_{Q_k^L}(\sigma_{q_k}(\Delta_{Q_k^R}(P'))) \cdot \Delta_{q_k}(p_d) = \Delta_Q(P' \cdot p_d) = \Delta_Q(P).\]
Also by induction, the left side of \eqref{eq-pairing} equals
\[
\sum_{k=1}^d(P')\nabla_{ \sigma_{p_d}Q_k^L}\nabla_{Q_k^R}\cdot (p_d)\nabla_{q_k} = (P' \cdot p_d)\nabla_Q = (P)\nabla_Q.
\] 
All that remains is to show the adjointness properties: \begin{align*}
\langle P_1P_2, Q\rangle &= \Delta_{P_1P_2}(Q) = \Delta_{P_1}(\Delta_{P_2}(Q)) = \langle P_1, \Delta_{P_2}(Q)\rangle\\[2mm]
\langle P_1P_2, Q\rangle &= (Q)\nabla_{P_1P_2} = ((Q)\nabla_{P_1})\nabla_{P_2} = \langle P_2, (Q)\nabla_{P_1} \rangle. \qedhere
\end{align*}
\end{proof}

\begin{ex}
By Example~\ref{ex-delta} and Proposition~\ref{prop-pairing}, 
\[\langle x_{12}x_{13}x_{12}, x_{12}x_{23}x_{31} \rangle = \langle x_{12}x_{13}, \Delta_{12}(x_{12}x_{23}x_{31}) \rangle = \langle x_{12}x_{13}, x_{23}x_{31} \rangle.\]
Similarly,
\[\langle x_{12}x_{13}, x_{23}x_{31}\rangle = \langle (x_{12}x_{13})\nabla_{23}, x_{31} \rangle = \langle -x_{13}, x_{31} \rangle = 1.\]
\end{ex}

Note that if $P$ and $Q$ are both homogeneous with respect to both the usual degree and $\S_n$-degree, then $\langle P, Q \rangle = 0$ unless $P$ and $Q$ have the same degree and $\sigma_P = \sigma_Q^{-1}$.

\theoremstyle{plain}
\newtheorem*{conj-nichols}{Conjecture \ref{conj-nichols}}
\begin{conj} \label{conj-nichols} \cite{MS}
The bilinear form $\langle\cdot,\cdot\rangle$ is nondegenerate on $\FK_n$.
\end{conj}
This conjecture is equivalent to $\FK_n$ being a special type of braided Hopf algebra called a Nichols algebra, which in this case is the quotient of $\FK_n$ by the kernel of the bilinear form (see \cite{AS}). It is known that Conjecture~\ref{conj-nichols} holds for $n \leq 5$ \cite{Grana}.

\subsection{Coproduct}

The Fomin-Kirillov algebra has the structure of a braided Hopf algebra. This was noted in \cite{MS} and can be derived from the Hopf algebra structure of the twisted version of the algebra described in \cite{FP}. For more information about braided Hopf algebras, see \cite{AS}. We describe only the coproduct here, as we will need it later.

The tensor product $\FK_n \otimes \FK_n$ has a braided product structure given by \[(P_1 \otimes Q_1)(P_2 \otimes Q_2) = (P_1 \sigma_{Q_1}(P_2)) \otimes (Q_1Q_2)\] for monomials $P_1$, $P_2$, $Q_1$, and $Q_2$. Then the coproduct $\Delta\colon \FK_n \to \FK_n \otimes \FK_n$ is defined to be the braided homomorphism such that $\Delta(x_{ij}) = x_{ij} \otimes 1 + 1 \otimes x_{ij}$.

Let $\FK_n^\vee$ be the graded dual of $\FK_n$, that is, the direct sum of the duals of each graded piece of $\FK_n$.  Then $\Delta$ defines an action of $\FK_n^\vee$ on $\FK_n$ as follows: if $p^\vee \in \FK_n^\vee$ and $Q \in \FK_n$, we let $p^\vee \act Q = \sum p^\vee(Q_{(1)}^i) \cdot Q_{(2)}^i$, where $\Delta(Q) = \sum Q_{(1)}^i \otimes Q_{(2)}^i$.

\begin{ex}
We calculate $\Delta(x_{12}x_{23})$ to be
\[(x_{12} \otimes 1+1 \otimes x_{12})(x_{23} \otimes 1 + 1 \otimes x_{23}) = x_{12}x_{23} \otimes 1 + x_{12}\otimes x_{23} + x_{13} \otimes x_{12} + 1 \otimes x_{12}x_{23}.\] Note that $(1 \otimes x_{12})(x_{23} \otimes 1) = x_{13} \otimes x_{12}$ due to the braiding.

Let $\{x_{ij}^\vee\} \subset \FK_n^\vee$ be the dual basis to $\{x_{ij}\} \subset \FK_n^1$. Then $x_{13}^\vee\act x_{12}x_{23}=x_{12}$.
\end{ex}

\begin{rmk}
If $x_{ij}^\vee$ is an element of the dual basis as above, then $x_{ij}^\vee \act Q = \rev ((\rev Q)\nabla_{ij}$).
\end{rmk}

\subsection{Properties of subalgebras}

It is important to note how our subalgebras $\FK_G$ behave with respect to the operators and bilinear form described above. The following lemma follows easily from the definitions of these operators.

\begin{lemma} \label{lemma-delta}
\begin{enumerate}[(a)]
\item If $\overline{ij} \not \in G$, then $\Delta_{ij}(\FK_G)=0$.
\item For any $\nabla_{ij}$ and any graph $G$, $(\FK_G) \nabla_{ij} \subset \FK_G$.
\item The coproduct $\Delta$ sends any element of $\FK_G$ into $\FK_n \otimes \FK_G$.
\item The left action of $\FK_n^\vee$ on $\FK_n$ restricts to an action on $\FK_G$.
\item If $H$ is a subgraph of $G$, then $\Delta(\FK_H^+\FK_G) \subset \FK_H^+\FK_n \otimes \FK_G + \FK_n \otimes \FK_H^+\FK_G$.
\end{enumerate}
\end{lemma}
In the language of braided Hopf algebras, Lemma~\ref{lemma-delta}(c) says that $\FK_G$ is a \emph{left coideal subalgebra} of $\FK_n$. Likewise, Lemma~\ref{lemma-delta}(e) implies that $\FK_H^+\FK_n$ is a \emph{coideal} of $\FK_n$.

One important consequence is the following.
\begin{lemma} \label{lemma-orthogonal}
Let $G_1$ be a graph on $n$ vertices and $G_2$ its complement. Then the left ideal $\FK_n\FK_{G_1}^+$ is orthogonal to $\FK_{G_2}$ with respect to $\langle \cdot, \cdot \rangle$.
\end{lemma}
\begin{proof}
This follows using the adjointness of right multiplication by $x_{ij}$ and the left action of $\Delta_{ij}$ for $\overline{ij} \in G_1$ together with Lemma~\ref{lemma-delta}(a).
\end{proof}

\subsection{Finite dimensionality} \label{sec-finitedim}
In this section, we will show that if $\FK_G$ is finite-dimensional, then its Hilbert series must have symmetric coefficients. (This was proven for $\FK_n$ in \cite{MS}.) We begin with a definition motivated by the theory of Coxeter groups.

\begin{definition}
For $w\in \FK_n$, the \emph{right descent set} of $w$, denoted $R(w)$, is the graph containing edge $\overline{ij}$ whenever $wx_{ij} = 0$. Similarly, define the \emph{left descent set} $L(w)$ as the graph containing $\overline{ij}$ whenever $x_{ij}w=0$.
\end{definition}
We now use these descent sets to prove a key lemma regarding the action of $\FK_n^\vee$ on $\FK_n$.

\begin{lemma} \label{lemma-integral}
For all $P \in \FK_n$, $\FK_n^\vee \act P$ is a left $\FK_{L(P)}$-module.
\end{lemma}
\begin{proof}
Let $Q=q^\vee \act P$ be an arbitrary element of $\FK_n^\vee \act P$. For $\overline{ij} \in L(P)$, we compute $\sigma_{ij}q^\vee \act x_{ij}P$, where $\sigma_{ij}q^\vee \in \FK_n^\vee$ is given by $(\sigma_{ij}q^\vee)(R) = q^\vee(\sigma_{ij}R)$ for all $R \in \FK_n$. If $\Delta(P) = \sum_k P_{(1)}^k \otimes P_{(2)}^k$, then
\begin{align*}
\Delta(x_{ij}P) &= \Delta(x_{ij})\Delta(P)\\
&= (x_{ij}\otimes 1 + 1 \otimes x_{ij}) \cdot \sum_k P_{(1)}^k \otimes P_{(2)}^k\\
&= \sum_k x_{ij}P_{(1)}^k \otimes P_{(2)}^k + \sum_k \sigma_{ij}P_{(1)}^k \otimes x_{ij}P_{(2)}^k.
\end{align*}
Thus
\begin{align*}
\sigma_{ij}q^\vee \act x_{ij}P &= \sum_k (\sigma_{ij}q^\vee)(x_{ij}P^k_{(1)})\cdot P_{(2)}^k + \sum_k (\sigma_{ij}q^\vee)(\sigma_{ij}P_{(1)}^k)\cdot x_{ij}P_{(2)}^k\\
&=-\sum_kq^\vee(x_{ij}\sigma_{ij}P^k_{(1)}) \cdot P^k_{(2)} + x_{ij}\cdot \sum_k q^\vee(P_{(1)}^k) \cdot P_{(2)}^k\\
&=-r^\vee \act P + x_{ij}Q,
\end{align*}
where $r^\vee \in \FK_n^\vee$ is given by $r^\vee(R) =q^\vee(x_{ij}\sigma_{ij}R)$. Since $\overline{ij} \in L(P)$, $x_{ij}P=0$, so  $x_{ij}Q = r^\vee \act P$.
\end{proof}
One can similarly show that $\FK_n^\vee \act P$ is a right $\FK_{R(P)}$-module. (See Proposition~\ref{prop-monicmcr} for a similar calculation.)

As a direct consequence of Lemma~\ref{lemma-integral}, we have the following result.

\begin{prop} \label{prop-monic}
Suppose that $G \subset L(P)$ for some $P \in \FK_n$. Then $\FK_G \subset \FK_n^\vee \act P$, and $\FK_G$ is finite-dimensional. If $P \in \FK_G$, then $\FK_n^\vee \act P = \FK_G$, and $P$ spans the top degree part of $\FK_G$.
\end{prop}
\begin{proof}
We may assume that $P$ is homogeneous of degree $d$. Then the only component of $\Delta(P)$ that lies in $\FK_n^d \otimes \FK_n^0$ is $P \otimes 1$. Thus $1 \in \FK_n^\vee \act P$. By Lemma~\ref{lemma-integral}, we must have that $\FK_G \subset \FK_n^\vee \act P$. But clearly every element of $\FK_n^\vee \act P$ has degree at most $d$, so $\FK_G$ has bounded degree and is therefore finite-dimensional.

If $P \in \FK_G$, then by Lemma~\ref{lemma-delta}, $\FK_n^\vee \act P \subset \FK_G$, so we must have $\FK_n^\vee \act P = \FK_G$. But the highest degree part of $\FK_n^\vee \act P$ has degree $d$, and since the homogeneous part of $\Delta(P)$ in $\FK_n^0 \otimes \FK_n^d$ is $1 \otimes P$, it follows that $P$ spans $\FK_G^d$.
\end{proof}

We can now prove the main theorem of this section.

\begin{thm} \label{thm-finitedim}
Let $G$ be a graph such that $\FK_G$ is finite-dimensional. Then
\begin{enumerate}[(a)]
\item the top degree component of $\FK_G$ is spanned by a single monomial $w_0^G$ of degree $d_0$;
\item for any $Q \in \FK_G$, there exists $p^\vee \in \FK_n^\vee$ such that $Q = p^\vee \act w_0^G$;
\item for any nonzero $Q \in \FK_G$, there exists a monomial $P \in \FK_G$ such that $PQ = w_0^G$;
\item the coefficients of the Hilbert series $\Hilb_G(t)$ are symmetric; 
\item the subwords of $w_0^G$ span $\FK_G$; and
\item $\rev(w_0^G) = \pm w_0^G$.
\end{enumerate}
\end{thm}
\begin{proof}
These all follow easily from Proposition~\ref{prop-monic}:

For (a), any element $w_0^G$ of top degree $d_0$ satisfies $x_{ij}w_0^G=0$ for all $\overline{ij} \in G$, so $w_0^G$ spans $\FK_G^{d_0}$.

For (b), this is equivalent to $\FK_n^\vee \act w_0^G = \FK_G$.

For (c), let $P \in \FK_G$ be a maximum degree monomial such that $PQ$ is nonzero. Then $x_{ij}PQ=0$ for all $\overline{ij} \in G$, so $PQ$ must be a multiple of $w_0^G$.

For (d), the bilinear form $\FK_G^d \otimes \FK_G^{d_0-d} \to \QQ$ that sends $P \otimes Q$ to the coefficient of $w_0^G$ in $PQ$ is nondegenerate by (c), so $\FK_G^d$ and $\FK_G^{d_0-d}$ have the same dimension.

For (e), any element of $\FK_n^\vee \act w_0^G = \FK_G$ lies in the span of the subwords of $w_0^d$ by the definition of the $\act$ action. 

For (f), $\rev(w_0^G) = cw_0^G$ for some constant $c$, and since $\rev$ is an involution, $c=\pm 1$.
\end{proof}

\begin{ex}
For $G=K_{1,3}$, as in Example~\ref{ex-star3}, the lexicographically minimal choice for $w_0^G$ is $\e{abacabac}$. For $G=K_{1,4}$, as in Example~\ref{ex-star4}, the lexicographically minimal choice is $w_0^G=\e{abacabacdabacabacdabadcabacd}$.
\end{ex}

\begin{rmk}
Theorem~\ref{thm-finitedim} implies that $\FK_G$ is a \emph{Frobenius algebra} when finite-dimensional: its Frobenius form is the bilinear form given in the proof of part (d).
\end{rmk}

The results in Theorem~\ref{thm-finitedim} are analogues of known results about finite Coxeter groups. See Section~\ref{sec-coxeter} for more discussion of this relationship.


\section{Tensor product decomposition}

\subsection{Subgraphs}
We begin this section by describing the relationship between $\FK_G$ and $\FK_H$ when $H$ is a subgraph of $G$.

\begin{thm} \label{thm-subgraph}
Let $H$ be a subgraph of $G$. Then $\FK_G$ is a free (left or right) $\FK_H$-module. Specifically, $\FK_G \cong \FK_H \otimes (\FK_G/\FK_H^+\FK_G)$ as left $\FK_H$-modules and $\FK_G \cong (\FK_G/\FK_G\FK_H^+) \otimes \FK_H$ as right $\FK_H$-modules.
\end{thm}
\begin{proof}
We prove just that $\FK_G$ is a free left $\FK_H$-module; the other result follows by passing to the opposite algebra. Let $I=\FK_H^+\FK_G$, and let the projection map be $\pi\colon \FK_G\to \FK_G/I$.

We first claim that if $f\colon \FK_G/I \to \FK_G$ is any degree-preserving ($\QQ$-linear) section and $\mu$ is the multiplication map, then $\varphi = \mu \circ (\idelm \otimes f) \colon \FK_H \otimes (\FK_G/I) \to \FK_G$ is surjective. We will prove by induction that $\FK_G^d$ lies in the image of $\varphi$. Note that the image of $\varphi$ is clearly closed under left multiplication by $\FK_H$. Then if $\FK_G^d$ lies in the image, so does the degree $d+1$ part of $I$. Since any element of $\FK_G^{d+1}$ differs from an element in the image of $f$ by an element of $I$ of degree $d+1$, it follows that $\FK_G^{d+1}$ also lies in the image, completing the induction.

Next we show that $\varphi$ is injective. Choose bases $\{h_i\}$ of $\FK_H$ and $\{\bar{g}_j\}$ of $\FK_G/I$, and suppose that $\varphi(\sum_{i,j} c_{ij} h_i \otimes \bar{g}_j)= \sum_{i,j} c_{ij} h_ig_j= 0$ for some constants $c_{ij}$ not all zero, where $g_j = f(\bar{g}_j)$. By restricting to the degree $d$ part, we may assume that $\deg h_i + \deg \bar{g}_j = d$ for all $i$ and $j$.

Find $i'$ such that some $c_{i'j}$ is nonzero with $h_{i'}$ of minimum degree $d'$. Let $\{h_{i}^\vee \mid \deg h_i \leq d\} \subset \FK_H^\vee$ be the dual basis to $\{h_i \mid \deg h_i \leq d\} \subset \FK_H$, and extend each $h_i^\vee$ to an element of $\FK_n^\vee$ arbitrarily. We claim that
\[0 = \pi(h_{i'}^\vee \act \textstyle\sum_{i,j} c_{ij}h_ig_j) = \textstyle\sum_j c_{i'j}\bar{g}_j.\]
This will be a contradiction since the $\bar{g}_j$ are linearly independent.

To see why the claim is true, consider any term $c_{ij}h_ig_j$ with $c_{ij}$ nonzero. Then $\Delta(h_ig_j) = \Delta(h_i) \cdot \Delta(g_j)$. Since $\FK_n \otimes I$ is a right ideal of $\FK_n \otimes \FK_G$ and $\Delta(h_i)$ is congruent to $h_i \otimes 1$ modulo $\FK_n \otimes I$, we find that
\[\pi(h_{i'}^\vee \act h_ig_j) = (h_{i'}^\vee \otimes \pi)(\Delta(h_ig_j))= (h_{i'}^\vee \otimes \pi)((h_i \otimes 1) \cdot \Delta(g_j)).\]
Since $\deg h_i \geq d'$, $(h_i \otimes 1) \cdot \Delta(g_j)$ can only have a nonzero component in $\FK_n^{d'} \otimes \FK_G^{d-d'}$ when $\deg h_i = d'$, in which case this component is $(h_i \otimes 1)(1 \otimes g_j) = h_i \otimes g_j$. Applying $h_{i'}^\vee \otimes \pi$ then gives 0 unless $i = i'$, in which case it gives $\bar{g}_j$, as desired. This completes the proof.
\end{proof}
\begin{cor}
Let $H$ be a subgraph of $G$. Then $\Hilb_H(t)$ divides $\Hilb_G(t)$, and their quotient has positive coefficients.
\end{cor}
\begin{proof}
The quotient is the Hilbert series of $\FK_G/\FK_G\FK_H^+$.
\end{proof}
\begin{rmk}
As noted in Lemma~\ref{lemma-delta}(c), $\FK_G$ is a left coideal subalgebra of the braided Hopf algebra $\FK_n$. Compare Theorem~\ref{thm-subgraph} to the results of \cite{Masuoka}, which gives several conditions that imply that a Hopf algebra is a free module over a (left) coideal subalgebra.
\end{rmk}

In light of Theorem~\ref{thm-subgraph}, we make the following definition.

\begin{defn} \label{defn-mcr}
A subset $M \subset \FK_G$ is a set of left (resp. right) \emph{minimal coset representatives} for $\FK_H$ if it is a basis of $\FK_G$ as a right (resp. left) $\FK_H$-module.
\end{defn}
Equivalently, by Theorem~\ref{thm-subgraph}, the projections of left (resp. right) minimal coset representatives give a basis of $\FK_G/\FK_G\FK_H^+$ (resp. $\FK_G/\FK_H^+\FK_G$). For this reason, we will sometimes abuse terminology and consider left minimal coset representatives to be elements of $\FK_G/\FK_G\FK_H^+$.

\begin{ex}
Let $G=A_3$ be the path with two edges as in Example~\ref{ex-a3}, and let $H$ be the subgraph containing only edge $\e{a}$. Then since $\FK_H$ has basis $\{\idelm, \e{a}\}$ and $\FK_G$ has basis $\{\idelm, \e{a}, \e{b}, \e{ab}, \e{ba}, \e{aba}\}$, a set of left minimal coset representatives is $\{\idelm, \e{b}, \e{ab}\}$.
\end{ex}

We use the term ``minimal coset representatives'' by analogy to the case of Coxeter groups and their corresponding nil-Coxeter algebras (see Section~\ref{sec-coxeter} for more details). We will typically take the elements of $M$ to be represented by monomials.

Using Theorem~\ref{thm-subgraph}, we can prove the following lemma, which will be useful in the next section.
\begin{lemma} \label{lemma-intersect}
Let $H$ be a subgraph of $G$. Then $\FK_H^+\FK_n \cap \FK_G = \FK_H^+\FK_G$.
\end{lemma}
\begin{proof}
Let $M$ and $N$ be sets of right minimal coset representatives for $\FK_H$ in $\FK_G$ and for $\FK_G$ in $\FK_n$, respectively. Then any element $x \in \FK_n$ can be written uniquely in the form $x=\sum h_{m,n}mn$, where $h_{m,n} \in \FK_H$, $m \in M$, and $n \in N$. Thus $\{mn \mid m \in M, n \in N\}$ is a set of right minimal coset representatives for $\FK_H$ in $\FK_n$. If $x \in \FK_H^+\FK_n$, then each $h_{m,n}$ has positive degree, while if $x \in \FK_G$, then $h_{m,n}=0$ unless $n$ is a constant. Thus if both hold, then we can write $x=\sum h_{m,n}m$ with $h_{m,n} \in \FK_H^+$, so $x \in \FK_H^+\FK_G$.
\end{proof}

\subsection{Finite rank}
In the event that $\FK_G$ has finite rank as an $\FK_H$-module, we can show that there is essentially a unique minimal coset representative of maximum degree. (If $\FK_G$ itself is finite-dimensional, then this follows from Theorem~\ref{thm-finitedim}.) The general result will follow from the following proposition, akin to Proposition~\ref{prop-monic}.

\begin{prop} \label{prop-monicmcr}
Let $H$ be a subgraph of $G$. Suppose $P \in \FK_G$ such that $P \not\in \FK_H^+\FK_G$ but $Px_{ij} \in \FK_H^+\FK_G$ for all $x_{ij} \in \FK_G$. Then $P$ spans the top degree part of $\FK_G/\FK_H^+\FK_G$. 
\end{prop}
\begin{proof}
Let $(\FK_n^\vee)^H$ be the set of all $q^\vee \in \FK_n^\vee$ such that $q^\vee(\FK_H^+\FK_n) = 0$. By Lemma~\ref{lemma-delta}(e), $q^\vee \act \FK_H^+\FK_G \subset \FK_H^+\FK_G$, so $(\FK_n^\vee)^H$ gives a left action $\act$ on $\FK_G/\FK_H^+\FK_G$.

We may assume that $P$ is homogeneous of degree $d$. We claim that $(\FK_n^\vee)^H \act P$ spans $\FK_G/\FK_H^+\FK_G$. First, since $P\not\in \FK_H^+\FK_G$, by Lemma~\ref{lemma-intersect}, $P \not\in \FK_H^+\FK_n$, so there exists a homogeneous element $q^\vee \in (\FK_n^\vee)^H$ such that $q^\vee(P)=1$. Then since the component of $\Delta(P)$ in $\FK_n^d \otimes \FK_G^0$ is $P \otimes 1$, $q^\vee\act P=1$, so $1 \in (\FK_n^\vee)^H \act P$. Then the claim will follow if we can show that the span of $(\FK_n^\vee)^H \act P$ in $\FK_G/\FK_H^+\FK_G$ is a right $\FK_G$-module.

Let $Q = q^\vee \act P \in (\FK_n^\vee)^H \act P$, and let $x_{ij}\in \FK_G$. Write $\Delta(P) = \sum_k P^k_{(1)} \otimes P^k_{(2)}$, so that
\[\Delta(Px_{ij}) = \sum_k P^k_{(1)}\sigma_{P^k_{(2)}}(x_{ij}) \otimes P^k_{(2)} + \sum_k P^k_{(1)} \otimes P^k_{(2)}x_{ij}.\]

Then
\begin{align*}
q^\vee \act (Px_{ij}) &= \sum_k q^\vee(P^k_{(1)}\sigma_{P^k_{(2)}}(x_{ij})) \cdot P^k_{(2)} + \sum_k q^\vee(P^k_{(1)}) \cdot P^k_{(2)}x_{ij}\\
&= -r^\vee \act P + Qx_{ij},
\end{align*}
where $r^\vee \in \FK_n^\vee$ is defined (for $\S_n$-homogeneous $R$) by $r^\vee(R) = -q^\vee(R \sigma_R^{-1}\sigma_P(x_{ij}))$. If $R$ lies in $\FK_H^+\FK_n$, then so does $R \sigma_R^{-1}\sigma_P(x_{ij})$. Hence $q^\vee \in (\FK_n^\vee)^H$ implies $r^\vee \in (\FK_n^\vee)^H$. Then since $Px_{ij} \in \FK_H^+\FK_G$, it follows that $r^\vee \act P$ and $Qx_{ij}$ are congruent modulo $\FK_H^+\FK_G$.

Thus $(\FK_n^\vee)^H \act P$ spans $\FK_G/\FK_H^+\FK_G$. Since the component of $\Delta(P)$ in $\FK_n^0 \otimes \FK_G^d$ is $1 \otimes P$, the top degree part of $(\FK_n^\vee)^H \act P$ is spanned by $P$, which gives the result.
\end{proof}

As an easy consequence, we get the following theorem analogous to Theorem~\ref{thm-finitedim}.
\begin{thm}
Let $H$ be a subgraph of $G$ such that $\FK_G$ has finite rank as an $\FK_H$-module, and let $M$ be a set of right minimal coset representatives. Then 
\begin{enumerate}[(a)]
\item $M$ has a unique element $m_0$ of top degree;
\item for any $m \in M$, $m = q^\vee \act m_0$ in $\FK_G/\FK_H^+\FK_G$ for some $q^\vee \in \FK_n^\vee$ with $q^\vee(\FK_H^+\FK_n) = 0$; and
\item for any $m \in M$, there exists $g \in \FK_G$ such that $m_0 = mg$ in $\FK_G/\FK_H^+\FK_G$. 
\end{enumerate}
\end{thm}
\begin{proof}
By Proposition~\ref{prop-monicmcr}, the only $m \in M$ such that $mx_{ij} \in \FK_H^+\FK_G$ for all $x_{ij} \in \FK_G$ has maximum degree in $M$ (which exists since $\FK_G$ has finite rank), and this element spans the top degree of $\FK_G/\FK_H^+\FK_G$ so must be unique. Parts (b) and (c) then follow as in the proof of Theorem~\ref{thm-finitedim}(b) and (c).
\end{proof}

\subsection{Complementary graphs} \label{sec-complementary}

In some cases, the tensor product decomposition described in Theorem~\ref{thm-subgraph} is particularly simple.

\begin{thm} \label{thm-tensor}
Let $G$ be a graph, and let $G_1$ and $G_2$ be complementary subgraphs of $G$ such that any two vertices in the same connected component of $G_2$ have the same neighbors in $G_1$. Then the multiplication map $\mu\colon \FK_{G_1} \otimes \FK_{G_2} \to \FK_G$ is an isomorphism of $\FK_{G_1}$-$\FK_{G_2}$-bimodules. In particular, $\Hilb_{G_1}(t)\cdot \Hilb_{G_2}(t) = \Hilb_G(t)$.
\end{thm}

As a special case of this theorem, we have the following corollary.

\begin{cor} \label{cor-tensor}
Let $G_1$ be a complete multipartite graph on $n$ vertices, and let $G_2$ be its complement, a disjoint union of complete graphs. Then $\FK_n \cong \FK_{G_1} \otimes \FK_{G_2}$.
\end{cor}
The case when $G_1 = K_{1, n-1}$ and $G_2 = K_{n-1}$ was proven in \cite{FP, MS}.

\begin{proof}[Proof of Theorem~\ref{thm-tensor}]
By our choice of $G_1$ and $G_2$, if $\overline{ij} \in G_1$ and $\overline{jk} \in G_2$, then $\overline{ik} \in G_1$.

We first show that $\mu$ is surjective. By Theorem~\ref{thm-subgraph}, it suffices to show that the map $\FK_{G_1} \to \FK_G/\FK_G\FK_{G_2}^+$ is surjective. Choose a monomial $p_1\cdots p_d \in \FK_G^d$. We show by induction on $d$ that it lies in the image of this map. Since the image is closed under left multiplication by $\FK_{G_1}$, we are done if $p_1 \in \FK_{G_1}$.

Then suppose $p_1 \in \FK_{G_2}$. By induction, we may assume that $p_2 \in \FK_{G_1}$. If $p_1$ commutes with $p_2$, then $p_1p_2 \cdots p_d = p_2 p_1 \cdots p_d$, and then we are again done because $p_2 \in \FK_{G_1}$. Otherwise, if $p_1 = x_{jk}$ and $p_2=x_{ij}$, then rewrite $x_{jk}x_{ij} = x_{ij}x_{ik} + x_{ik}x_{jk}$. Since $x_{ij}, x_{ik} \in \FK_{G_1}$, we are again done by induction.

To show that $\mu$ is injective, by Theorem~\ref{thm-subgraph}, it suffices to show that the map $\FK_{G_2} \to \FK_G/\FK_{G_1}^+\FK_G$ is injective, that is, that $\FK_{G_2}$ intersects $\FK_{G_1}^+\FK_G$ trivially. But this holds because the $\supp$-degree of any $\supp$-homogeneous element of $\FK_{G_2}$ has each part contained in a connected component of $G_2$, but this is not the case for any (nonzero) element of $\FK_{G_1}^+\FK_G$.
\end{proof}

In fact, it appears that the class of complementary graphs $G_1$ and $G_2$ for which $\FK_G \cong \FK_{G_1} \otimes \FK_{G_2}$ is much more general than Corollary~\ref{cor-tensor} implies. Using Theorem~\ref{thm-tensor} and computations for small graphs, we can prove the following partial result on when such a tensor product decomposition holds.

\begin{cor} \label{cor-complementary}
Let $G_1$ be a graph on $n$ vertices and $G_2$ its complement. The class of graphs $G_1$ for which $\FK_n \cong \FK_{G_1} \otimes \FK_{G_2}$ holds (as in Theorem~\ref{thm-tensor}) contains all graphs with at most five vertices and is closed under disjoint unions and complementation.
\end{cor}
\begin{proof}
We checked the claim for graphs with at most five vertices using \texttt{bergman}. Closure under complementation follows since $\FK_n$ and all $\FK_G$ are isomorphic to their opposite algebras.

For disjoint unions, suppose that the tensor product decomposition exists for graphs $G_1 \cup G_2 = K_m$ and $H_1 \cup H_2 = K_n$. Then the complement of $G_1+H_1$ in $K_{m+n}$ is $L=(G_2+H_2) \cup K_{m,n}$. By Theorem~\ref{thm-tensor}, 
\[
\FK_{m+n}
\cong \FK_{K_m+K_n} \otimes \FK_{K_{m,n}}
\cong \FK_{G_1+H_1} \otimes \FK_{G_2+H_2}\otimes \FK_{K_{m,n}}
\cong \FK_{G_1+H_1}  \otimes \FK_L.\qedhere\]
\end{proof}

However, the tensor product decomposition does not hold in general.

\begin{figure}
\begin{tabular}{ccc}
		\begin{tikzpicture}[scale=0.7]
			\node[v] (1) at (-0.5,0){};
			\node[v] (2) at (1,0){};
			\node[v] (3) at (2,1){};
			\node[v] (4) at (2,-1){};
			\node[v] (5) at (3, 0){};
			\node[v] (6) at (4.5, 0){};
			\draw(1) node[above]{1} -- (2)node[above]{2} -- (3)node[above]{3} -- (5)node[above]{5} -- (6)node[above]{6} (2)--(4)node[below]{4}--(5) (2)--(5);
		\end{tikzpicture}
&\quad&
		\begin{tikzpicture}[scale=0.7]
			\node[v] (1) at (-0.5,0){};
			\node[v] (2) at (1,0){};
			\node[v] (3) at (2,1){};
			\node[v] (4) at (2,-1){};
			\node[v] (5) at (3, 0){};
			\node[v] (6) at (4.5, 0){};
			\draw(1) node[above]{2} -- (2)node[above]{6} -- (3)node[above]{3} -- (5)node[above]{1} -- (6)node[above]{5} (2)--(4)node[below]{4}--(5) (3)--(4) (2)--(5);
		\end{tikzpicture}\\
		\begin{tikzpicture}[scale=0.7]
			\node[v] (1) at (0,0){};
			\node[v] (2) at (18:1.5){};
			\node[v] (3) at (90:1.5){};
			\node[v] (4) at (162:1.5){};
			\node[v] (5) at (234:1.5){};
			\node[v] (6) at (306:1.5){};
			\draw (2) node[right]{2}--(3)node[above]{1}--(4)node[left]{5}--(5)node[below]{4}--(6)node[below]{3}--(2)--(1)node[below]{6}--(4);
		\end{tikzpicture}
&\quad&
		\begin{tikzpicture}[scale=0.7]
			\node[v] (1) at (0,0){};
			\node[v] (2) at (18:1.5){};
			\node[v] (3) at (90:1.5){};
			\node[v] (4) at (162:1.5){};
			\node[v] (5) at (234:1.5){};
			\node[v] (6) at (306:1.5){};
			\draw (2) node[right]{3}--(3)node[above]{1}--(4)node[left]{4}--(5)node[below]{2}--(6)node[below]{5}--(2)--(1)node[below]{6}--(4) (1)--(3);
		\end{tikzpicture}
\end{tabular}
\caption{\label{fig-counter} Either of the graphs on the left can be used for $G_1$ in Proposition~\ref{prop-counter}. On the right are their complements.}
\end{figure}
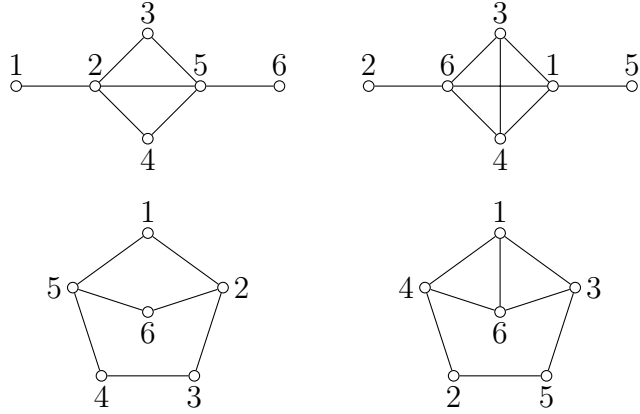

\begin{prop} \label{prop-counter}
Let $G_1$ be either of the two graphs on the left of Figure~\ref{fig-counter} and $G_2$ its complement (shown on the right). Then $\FK_6 \not \cong \FK_{G_1} \otimes \FK_{G_2}$.
\end{prop}
\begin{proof}
Note that $G_1$ has the property that its complement $G_2$ is isomorphic to $G_1$ but with an extra edge connecting two twin vertices. Hence by Theorem~\ref{thm-tensor}, $\Hilb_{G_2}(t) = \Hilb_{G_1}(t) \cdot (1+t)$. Then if $\FK_6 \cong \FK_{G_1} \otimes \FK_{G_2}$, we would be able to solve for $\FK_{G_1}$ using the first few terms of the Hilbert series for $\FK_6$ as given in Section~\ref{sec-hilbert} to find that
\begin{align*}
\Hilb_{G_1}(t) &= \left(\frac{\Hilb_6(t)}{1+t}\right)^{1/2}\\
&= \left(\frac{1 + 15t + 125t^2 + 765t^3 + 3831t^4 + 16605t^5 + 64432t^6 + 228855t^7+\cdots}{1+t}\right)^{1/2}\\
&= \left(1+14t+111t^2+654t^3+3177t^4+13428t^5+51004t^6+177851t^7+\cdots\right)^{1/2}\\
&= 1+7t+31t^2+110t^3+338t^4+938t^5+2408t^6+\tfrac{11623}{2}t^7+\cdots,
\end{align*}
which is impossible due to the coefficient of $t^7$.
\end{proof}

It would be interesting to try to classify for which graphs and their complements a tensor product decomposition as in Corollary~\ref{cor-tensor} holds.

\subsection{Computation}
 
We include a brief discussion of a method for computing minimal coset representatives. This method of computation was used to achieve some of the results in the next section as well as the conjectures we will present later.

Let $G$ be a graph and $e$ an edge not in $G$. Denote $G \cup \{e\}$ by $G'$. Suppose that we wish to compute a set of minimal coset representatives for $\FK_{G}$ in $\FK_{G'}$, that is, a basis for $\FK_{G'}/\FK_{G'}\FK_{G}^+$. Unfortunately, without prior knowledge of the relations of $\FK_{G'}$ (which would naively require an expensive noncommutative Gr\"obner basis calculation for $\FK_n$), one cannot easily determine whether an element of $\FK_{G'}$ lies in $\FK_{G'}\FK_{G}^+$.

We can, however, give a simple sufficient condition for an element of $\FK_{G'}$ not to lie in $\FK_{G'}\FK_{G}^+$. Let $H$ be the complement of $G'$ and $H' = H \cup \{e\}$ the complement of $G$. By Lemma~\ref{lemma-orthogonal}, every element of $\FK_{G'}\FK_{G}^+$ is orthogonal to $\FK_{H'}$. Hence, any element of $\FK_{G'}$ that pairs nontrivially with some element of $\FK_{H'}$ does not lie in $\FK_{G'}\FK_{G}^+$. But we need not even pair with all elements of $\FK_{H'}$: any element of $\FK_{H'}\FK_H^+$ is orthogonal to every element of $\FK_{G'}$. Hence we only need pair with elements that are linearly independent in $\FK_{H'}/\FK_{H'}\FK_H^+$.

This suggests the following algorithm to calculate linearly independent sets of minimal coset representatives for $\FK_G$ in $\FK_{G'}$ and for $\FK_H$ in $\FK_{H'}$ simultaneously:

\begin{alg}\label{alg-mcr}
Let $G$ and $H$ be graphs and $e$ an edge such that $G \sqcup H \sqcup \{e\}=K_n$. Let $G' = G \cup \{e\}$ and $H' = H \cup \{e\}$, and set $M^0 = N^0 = \{\idelm\}$. For $d \geq 0$:
\begin{itemize}
\item Construct a matrix with rows indexed by $x_{ij}p$ for $x_{ij} \in \FK_{G'}$ and $p \in M^d$, columns indexed by $x_{kl}q$ for $x_{kl} \in \FK_{H'}$ and $q \in N^d$, and entries $\langle x_{ij}p, x_{kl}q \rangle$.
\item Set $M^{d+1}$ to be the indices of a maximal set of linearly independent rows and likewise $N^{d+1}$ for columns.
\end{itemize}
Then $M^d$ and $N^d$ are subsets of degree $d$ minimal coset representatives for $\FK_{G}$ inside $\FK_{G'}$ and $\FK_H$ inside $\FK_{H'}$. (In other words, $M^d$ is linearly independent modulo $\FK_{G'}\FK_G^+$, and likewise $N^d$ modulo $\FK_{H'}\FK_H^+$.)
\end{alg}
\begin{rmk}
Algorithm~\ref{alg-mcr} can be used to give a lower bound on $\Hilb_{G'}(t)/\Hilb_G(t)$, but this will not in general be exact. However, it is usually quite accurate for small graphs. For instance, one can show that if Conjecture~\ref{conj-nichols} holds and $\FK_{H'} \otimes \FK_G \cong \FK_n$ (as in Theorem~\ref{thm-tensor}), then Algorithm~\ref{alg-mcr} will give a complete set of minimal coset representatives for $\FK_G$ in $\FK_{G'}$. In particular, the algorithm is exact for all graphs on at most five vertices.
\end{rmk}


\section{Coxeter groups and nil-Coxeter algebras} \label{sec-coxeter}

In this section, we will describe various ways in which the subalgebras $\FK_G$ share similar properties to Coxeter groups, or more specifically, to their nil-Coxeter algebras.

\subsection{Definitions}

We first recall the definition of (simply-laced) Coxeter groups and nil-Coxeter algebras as well as some of their basic properties. For further details, see \cite{hump} or \cite{BjornerBrenti}.

Let $D$ be a graph, which we will refer to in this context as a \emph{simply-laced Dynkin diagram}.

\begin{defn}
Given a simply-laced Dynkin diagram $D$, the \emph{Coxeter group} $(W, S)$ of $D$ is the group generated by $S$, the set of \emph{simple reflections} $s_i$ for each vertex $i$ of $D$, and relations
\begin{itemize}
	\item $s_i^2 = 1$ for any vertex $i$ of $D$;
	\item $s_is_j = s_js_i$ for nonadjacent vertices $i,j$ of $D$;
	\item $s_is_js_i = s_js_is_j$ for adjacent vertices $i,j$ of $D$ (called a \emph{braid relation}).
\end{itemize}
\end{defn}
\begin{defn}
Given an element $w \in W$, a \emph{reduced word} (or \emph{reduced expression}, or \emph{reduced decomposition}) $s_{i_1}s_{i_2}\cdots s_{i_\ell}$ for $w$ is a minimum length expression of $w$ as a product of generators $s_i$. The \emph{length} of $w$, denoted $\ell(w)$, is the length of any reduced word for $w$. We say that $w = u \cdot v$ is a \emph{reduced factorization} if $\ell(w) = \ell(u)+\ell(v)$.
\end{defn}

Given a Coxeter group, one can define the corresponding nil-Coxeter algebra as follows.

\begin{definition}
Given a simply-laced Dynkin diagram $D$, the \emph{nil-Coxeter algebra $\Nil$} is the associative algebra with a generator $t_i$ for each vertex $i$ of $D$ and relations
\begin{itemize}
	\item $t_i^2 = 0$ for any vertex $i$ of $D$;
	\item $t_{i}t_{j} = t_j t_i$ for nonadjacent vertices $i,j$ of $D$;
	\item $t_{i}t_{j}t_{i} = t_{j}t_{i}t_{j}$ for adjacent vertices $i,j$ of $D$.
\end{itemize}
\end{definition}
Note the similarity of this definition to that of $\FK_n$.

Given an element $w \in W$, we can define an element $t_w \in \Nil$ by choosing any reduced word $w = s_{i_1}s_{i_2} \cdots s_{i_\ell}$ and letting $t_w = t_{i_1}t_{i_2} \cdots t_{i_\ell}$. The element $t_w$ does not depend on the choice of reduced word.
\begin{prop}
The nil-Coxeter algebra $\Nil$ has basis $\{t_w \mid w \in W\}$ with multiplication given by
\[t_ut_v = \begin{cases} t_{uv},&\text{if $\ell(u)+\ell(v) = \ell(uv)$;}\\ 0,& \text{otherwise.}\end{cases}\]
\end{prop}
For this reason, we will sometimes abuse notation and use the same letters to refer to both elements of $W$ and elements of $\Nil$.

\subsection{Line graphs}

We briefly describe how to relate the subalgebra $\FK_G$ to a certain (twisted) nil-Coxeter algebra.

\begin{definition}
Given a graph $G$, let $L(G)$ denote the line graph of $G$, which we will think of as a simply-laced Dynkin diagram.
Given a directed graph  $G'$, the \emph{twisted nil-Coxeter algebra} of $L(G')$ is the associative algebra with a generator for each edge of $G'$ and relations
\begin{itemize}
    \item $\e{e} = \e{f}$ for edges $\e{e},\e{f}$ of $G'$ with the same ends and direction;
     \item $\e{e} = -\e{f}$ for any directed 2-cycle $\e{e},\e{f}$ in $G'$;
	\item $\e{e}^2 = 0$ for any edge $\e{e}$ of $G'$;
	\item $\e{ef} = \e{fe}$ for edges $\e{e}, \e{f}$ of $G'$ that do not share an end;
	\item $\e{efe} = \e{fef}$ for edges $\e{e}, \e{f}$ of $G'$ that form a directed path;
\item $\e{efe} = -\e{fef}$ for edges $\e{e}, \e{f}$ of $G'$ that share one end but do not form a directed path.
\end{itemize}
We refer to these last two relations as \emph{positive} and \emph{negative braid relations}, respectively.
\end{definition}
Note that we have abused notation slightly since the algebra depends on $G'$ and not just the line graph $L(G')$.

The nonzero monomials of the twisted nil-Coxeter algebra of $L(G')$ are in bijection with the nonzero monomials of the nil-Coxeter algebra of $L(G)$, where $G$ is the underlying simple undirected graph of $G'$.
Since the generators of $\FK_G$ satisfy the same versions of the braid relation satisfied by the generators of the twisted nil-Coxeter algebra, we have the following proposition.
\begin{proposition}\label{p Coxeter to FK}
For an undirected graph  $G$, let $G'$ be any directed graph whose underlying simple undirected graph is $G$.
The algebra $\FK_G$ is a quotient of the twisted nil-Coxeter algebra
of $L(G')$. Moreover, if the edges of $G$ can be directed so that the
  indegree and outdegree at each vertex are at most one, then $\FK_G$ is a quotient of the nil-Coxeter algebra
  of $L(G)$.
\end{proposition}

It is important to note that the twisted nil-Coxeter algebra of $L(G')$ is almost always infinite-dimensional since, among the Dynkin diagrams of finite type, only those of type $A_n$ are line graphs. Nevertheless, we will use Proposition~\ref{p Coxeter to FK} in our discussion of $\FK_{\tilde{A}_{n-1}}$, utilizing the fact that the line graph of a cycle is again a cycle.

\subsection{Analogy to Fomin-Kirillov algebras}

Many of the results proved about $\FK_G$ in the previous sections are analogues of the following facts about Coxeter groups and nil-Coxeter algebras. Throughout this section, let $D$ be a simply-laced Dynkin diagram, $(W, S)$ its Coxeter group, and $\Nil$ its nil-Coxeter algebra.

There is a notion of \emph{Bruhat order} in $W$ defined as follows. Let $T = \{wsw^{-1} \mid w \in W, s \in S\}$ be the set of \emph{reflections} of $W$. The \emph{Bruhat order $<$} of $W$ is the transitive closure of the relations $tw < w$ for $t\in T$ such that $\ell(tw) < \ell(w)$. The \emph{left weak order $<_L$} of $W$ is the transitive closure of the relations $sw <_L w$ for $s\in S$ such that $\ell(sw) < \ell(w)$.

An equivalent way of formulating Bruhat order is that $v \leq w$ if some reduced word for $w$ contains a substring that gives a reduced word for $v$; in fact, if this is true for some reduced word for $w$, then it is true for any reduced word for $w$. We can also formulate left weak order similarly: $v \leq_L w$ if there exists a reduced word for $w$ that has a final substring that is a reduced word for $v$ (but this depends on the reduced word for $w$).

The analogue of Bruhat order in the Fomin-Kirillov algebra setting is, for $P, Q \in \FK_G$:
\[
Q \leq P \text{\quad if } Q \in \FK_n^\vee \act P.
\]
From the definition of the action of $\FK_n^\vee$, it follows that if $Q \leq P$, then $Q$ lies in the span of the subwords of $P$.
\begin{rmk}
If Conjecture~\ref{conj-nichols} holds, then $Q \leq P$ if and only if $Q=(P)\nabla_x$ for some $x \in \FK_n$.
\end{rmk}
The analogue of left weak order in the Fomin-Kirillov algebra setting is, for $P, Q \in \FK_G$:
\[
Q \leq_L P \text{\quad if there exists } x \in \FK_G \text{ such that } xQ=P.
\]
Theorem~\ref{thm-finitedim} is then the analogue of the following facts about $W$ and $\Nil$: if $W$ is finite, then:
\begin{enumerate}[(a)]
\item there is a unique element $w_0 \in W$ of maximum length (so the top degree part of $\Nil$ is one-dimensional);
\item $w\leq w_0$ for all $w\in W$;
\item $w \leq_L w_0$ for all $w\in W$;
\item the Poincar\'e series of $W$ (and hence the Hilbert series of $\Nil$) is symmetric;
\item every element of $W$ is a subword of any reduced expression for $w_0$; and
\item $w_0 = w_0^{-1}$.
\end{enumerate}

\begin{rmk}
Although $Q \leq P$ is well-defined (independent of the choice of expression for $P$), the span of the subwords of $P$ is not well-defined in general. For instance, in $\FK_3$, $x_{12}x_{23}x_{12} = x_{12}x_{13}x_{23}$, but clearly $x_{13}$ does not lie in the span of the subwords of $x_{12}x_{23}x_{12}$. However, Theorem~\ref{thm-finitedim} shows that this notion is well-defined when $P = w_0^G$.

Likewise, in $\FK_G$, $Q \leq_L P$ does not imply $Q \leq P$ in general: for instance, $x_{13}x_{23} \leq_L x_{12}x_{13}x_{23}$, but $x_{13}x_{23} \not\leq x_{12}x_{13}x_{23}$. But again, Theorem~\ref{thm-finitedim} shows that these notions coincide when $P = w_0^G$.
\end{rmk}

The notion of descent sets introduced in Section~\ref{sec-finitedim} is also the direct analogue of a notion from Coxeter groups: the \emph{left descent set} $L(w)$ of an element $w \in W$ is the set of all $s_i \in S$ such that $\ell(s_iw)<\ell(w)$ (or in $\Nil$, $s_iw=0$). (Define $R(w)$ similarly.) Then Proposition~\ref{prop-monic} is the analogue of the following statement about Coxeter groups:
\begin{itemize}
\item  if $W$ has an element $w$ such that $L(w) = S$, then $W$ is finite-dimensional and $w$ is the longest element of $W$.
\end{itemize}
We can also now prove the following results about descent sets in $\FK_G$, which are again analogues of facts about Coxeter groups.
\begin{prop}
If $w \in \FK_G$, then $L(w),R(w)\subset G$.
\end{prop}
\begin{proof}
Suppose $e = \overline{ij} \in L(w)$. Then Corollary~\ref{cor-tensor} implies that $\FK_e \otimes \FK_H \cong \FK_n$, where $H=K_n \backslash \{e\}$. Thus left multiplication by $x_{ij}$ is injective on $\FK_H$, so since $x_{ij}w = 0$, we must have that $w \not\in \FK_H$. Thus $G \not\subset H$, so $e \in G$.
\end{proof}
\begin{proposition}\label{p descent set longword}
Let $w \in \FK_G$, and suppose $H \subset L(w)$. Then one can write $w = w_0^H g$ for some $g \in \FK_G$. (Similarly, if $H' \subset R(w)$, then $w = g' w_0^{H'}$ for some $g' \in \FK_G$.)
\end{proposition}
Recall that by Proposition~\ref{prop-monic}, $\FK_H$ is necessarily finite-dimensional.
\begin{proof}
By Theorem~\ref{thm-subgraph}, there is a unique expression for $w$ of the form  $w= \sum_{m \in M} h_mm$ for some  $h_m \in \FK_H$, where $M$ is a set of right minimal coset representatives for $\FK_H$ in $\FK_G$.
For any $\overline{ij} \in H$, since $\FK_G$ is a free left $\FK_H$-module,
$0=x_{ij}w=\sum_{m \in M} (x_{ij}h_m)m$ implies that $x_{ij} h_m=0$ for all $m\in M$. Hence by Proposition~\ref{prop-monic}, $h_m = c_mw_0^H$ for some constant $c_m$, so $w = w_0^H (\sum_{m \in M} c_mm)$.
\end{proof}
Along similar lines, we can prove the following result.
\begin{prop}
Let $w \in \FK_G$, and suppose $w_0^Hw = 0$ for some $H \subset G$ with $\FK_H$ finite-dimensional. Then $w \in \FK_H^+\FK_G$.
\end{prop}
\begin{proof}
As in Proposition~\ref{p descent set longword}, we can write $w= \sum_{m \in M} h_mm$. Multiplying by $w_0^H$ gives $0= \sum_{m \in M} (w_0^H h_m)m$. Since $M$ is a set of minimal coset representatives, $w_0^H h_m=0$ for all $m$. But then each $h_m$ has degree at least 1, so $w \in \FK_H^+\FK_G$.
\end{proof}

Finally, Coxeter groups have a notion of parabolic subgroups: for $J \subseteq S$, the \emph{parabolic subgroup} $W_J$ is the subgroup of $W$ generated by the set $J$. The pair $(W_J, J)$ is equal to the Coxeter group of the induced subgraph of $D$ with vertex set $J$. Let $\Nil_J$ denote the nil-Coxeter algebra of $W_J$.

The analogues of parabolic subgroups for $\FK_G$ are the subalgebras $\FK_H$ as $H$ ranges over subgraphs of $G$. However the fact about parabolic subgroups just mentioned does not hold---the minimal relations satisfied by the generators of $\FK_H$ are not easily determined from those of $\FK_G$. Despite this, we still have that Theorem~\ref{thm-subgraph} is the analogue of the following fact about parabolic subgroups:
\begin{itemize}
\item Each right (resp. left) coset $W_Jw$ (resp. $wW_J$) of $W_J$ contains a unique element of minimal length called a \emph{minimal coset representative}. Let $\leftexp{J}{W}$ (resp. $W^J$) denote the set of all such elements. Then $\Nil$ is a free left (resp. right) $\Nil_J$-module and as a left (resp. right) $\Nil_J$-module has basis  $\leftexp{J}{W}$ (resp.  $W^J$).
\end{itemize}

Although there exist many parallels between Coxeter groups and Fomin-Kirillov algebras, we do not yet have a good analogue for many of the geometric notions associated with Coxeter groups, such as reflections or root systems. For instance, for finite Coxeter groups, the length of the long word equals the number of positive roots in the corresponding root system, but we know of no combinatorial object that determines the maximum degree that occurs in $\FK_G$.

\section{Dynkin diagrams}

In this section, we will describe $\FK_G$ when $G$ is a (simply-laced) Dynkin diagram of finite type. In particular, we will show that any minimal relation of $\FK_G$ in these cases is of one of three types: it is either a \emph{quadratic relation} inherited from $\FK_n$, a \emph{braid relation} of the form $\e{aba}+\e{bab}=0$ between two edges that share a vertex, or a \emph{claw relation} of the form $\e{abca}+\e{bcab}+\e{cabc}=0$ among three edges that share a vertex. (See the examples in Section~\ref{sec-examples} and Lemma~\ref{lemma-cyclic}.)


\subsection{The case $\FK_{A_n}$}

Define $A_n$ to be the path on $n$ vertices (that is, the Dynkin diagram of type $A_n$).
Then, in accordance with Example~\ref{ex-a3}, we have the following theorem.

\begin{thm} \label{thm-a}
	The only minimal relations in $\FK_{A_n}$ are the quadratic and braid relations. Its Hilbert series is
	\[\Hilb_{A_n}(t) = [2][3][4] \cdots [n].\] In other words, $\FK_{A_n} \cong \Nil_n$, the \textit{nil-Coxeter algebra} of type $A_{n-1}$.
\end{thm}
This follows easily from the existence of the divided difference representation of $\FK_n$ given in \cite{FK}, but we present a different proof as it will be related to our investigation of $\FK_{\tilde{A}_{n-1}}$ later.

\begin{proof}
Since $\fWnil_{n}$ is defined using the quadratic and braid relations, $\FK_{A_n}$ is a quotient of $\fWnil_{n}$. 
We will show that the projection $\Theta\colon \fWnil_{n} \to \FK_{A_n}$ is an isomorphism by showing that the image under $\Theta$ of the basis $\{w\mid w \in \S_n\}$ of $\fWnil_{n}$ is linearly independent.
We prove this by the following pairing computation:
 \begin{equation*}\label{e pairing Sn}
    \text{ for $w,v\in \S_n$,} \qquad\qquad
   \langle \Theta(w), \rev(\Theta(v)) \rangle =
   \begin{cases}
     1 & \text{ if } w = v,\\
     0 & \text{ if } w\neq v.
   \end{cases}
 \end{equation*}

This pairing is zero unless the $\S_n$-degree of $\Theta(w)$ and $\Theta(v)$ are the same, which gives the second case.  To prove $\langle \Theta(w), \rev(\Theta(w)) \rangle = 1$, we induct on the length of $w$.
Let $P = \Theta(w) = p_1\cdots p_d$ be an expression for $\Theta(w)$ as a product of generators of $\FK_{A_n}$, and write $P = P^L_jp_jP^R_j$.
Then by Lemma  \ref{lemma-leibniz} and Proposition \ref{prop-pairing},
\[
\langle P,\rev (P)\rangle
 =\sum_{j=1}^d (p_j) (\sigma_{P^R_j}\nabla_{p_d}) \cdot \langle P^L_j P^R_j,
 \rev(P^L_d) \rangle.\]
 By the strong exchange condition for $\S_n$, there is a unique index $j$ for which $P^L_jP^R_j$ has the same $\S_n$-degree as $P^L_d$; clearly this index is $j=d$. Hence only the $j=d$ term in the sum can be nonzero, so we find that $\langle P, \rev(P) \rangle = \langle P^L_d, \rev(P^L_d)\rangle = 1$ by induction.
\end{proof}

\subsection{The case $\FK_{D_n}$}

The next simplest case is that of the Dynkin diagram of type $D_n$ (which we will simply call $D_n$). We denote the edges of $D_n$ by $\e{a}$, $\e{b}$, $\e{1}$, \dots, $\e{n-3}$ as shown in Figure~\ref{fig-d}. (We will use these labels to denote both the edges of the graph and the corresponding variables.)

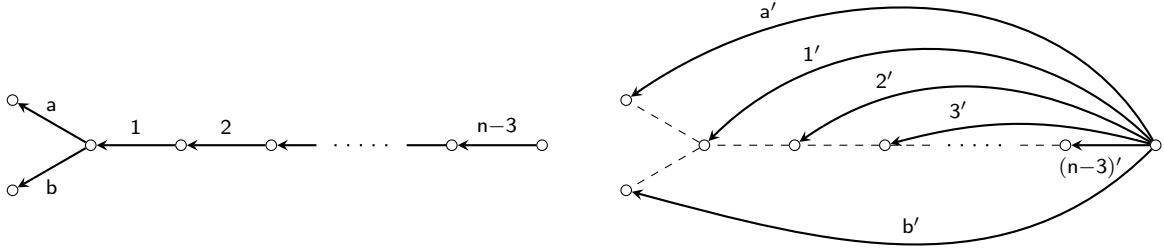
\begin{figure}
	\begin{center}
	\begin{tikzpicture}[scale = 1.2]
			\node[v] (a) at (150:1){};
			\node[v] (b) at (-150:1){};
			\node[v] (1) at (0,0){};
			\node[v] (2) at (1,0){};
			\node[v] (3) at (2,0){};
			\node[v] (4) at (4,0){};
			\node[v] (5) at (5,0){};
			\node at (0,-1.6){};
			\draw[thick, ->] (5)--node[above]{${\scriptstyle \e{n-3}}$}(4);
			\draw[thick] (4)--(3.5,0);
			\draw[thick, ->] (2.5,0)--(3);
			\draw[thick, ->] (3)--node[above]{${\scriptstyle \e{2}}$}(2);
			\draw[thick, ->] (2)--node[above]{${\scriptstyle \e{1}}$}(1);			
			\draw[thick, ->] (1)--node[above]{${\scriptstyle \e{a}}$}(a);
			\draw[thick, ->] (1)--node[below]{${\scriptstyle \e{b}}$}(b);
			\draw[thick, loosely dotted] (2.7,0)--(3.4,0);
		\end{tikzpicture}
	\qquad
	\begin{tikzpicture}[scale = 1.2]
			\node[v] (a) at (150:1){};
			\node[v] (b) at (-150:1){};
			\node[v] (1) at (0,0){};
			\node[v] (2) at (1,0){};
			\node[v] (3) at (2,0){};
			\node[v] (4) at (4,0){};
			\node[v] (5) at (5,0){};
			\node at (0,-1.6){};
			\draw[thick, ->] (5)--node[below, near end]{${\scriptstyle \e{(n-3)'}}$}(4);	
			\draw[thick, loosely dotted] (2.7,0)--(3.4,0);
			\draw[thick, ->] (5) to [out = 165, in=15, near end] node[above]{${\scriptstyle \e{3'}}$}(3);
			\draw[thick, ->] (5) to [out =150, in=35, near end] node[above]{${\scriptstyle \e{2'}}$}(2);
			\draw[thick, ->] (5) to [out =135, in=45, near end] node[above]{${\scriptstyle \e{1'}}$}(1);
			\draw[thick, ->] (5) to [out =120, in=35, near end] node[above]{${\scriptstyle \e{a'}}$}(a);
			\draw[thick, ->] (5) to [out =-135, in=-15] node[above]{${\scriptstyle \e{b'}}$}(b);
			\draw[dashed] (4)--(3.5,0) (2.5,0)--(3)--(2)--(1)--(a) (1)--(b);
	\end{tikzpicture}
	\end{center}
	\caption{\label{fig-d} On the left, the Dynkin diagram $D_n$ with edge labels. On the right, the star $K_{1,n-1}$ with primed edge labels to be used in Lemma~\ref{lemma-d3}.}
\end{figure}

The main result of this section is the following theorem. 
	
\begin{thm} \label{thm-d}
	The only minimal relations in $\FK_{D_n}$ are the quadratic, braid, and claw relations. Its Hilbert series is
	\[\Hilb_{D_n}(t) = [n][n-1] \cdot [4][6][8]\cdots[2n-4].\]
\end{thm}
  
In order to prove Theorem~\ref{thm-d}, we will construct a set of minimal coset representatives for $\FK_{D_{n-1}}$ in $\FK_{D_{n}}$. For $n \geq 3$, let
\begin{align*}
M_n=\{&\idelm, \quad \e{n-3}, \quad \e{(n-4)(n-3)}, \quad \dots, \quad \e{(n-3)\factR }, \\
&\e{a(n-3)\factR }, \quad \e{b(n-3)\factR },\quad \e{ab(n-3)\factR }, \quad  \e{ba(n-3)\factR }, \quad \e{aba(n-3)\factR }, \\
& \e{1aba(n-3)\factR },\quad \e{2\factL aba(n-3)\factR },\quad \dots,\quad  \e{(n-3)\factL  aba(n-3)\factR } \},
\end{align*}
where $\e{i\factR}=\e{12\cdots(i-1)i}$ and $\e{i\factL}=\e{i(i-1)\cdots 21}$.

Note that $M_n$ can be given the structure of a (graded) partially ordered set in which $m'$ covers $m$ if there exists an edge variable $\e{e}$ such that $m' = \e{e}\,m$. Thus we can label the edges in the Hasse diagram of $M_n$ by generators of $\FK_{D_n}$, and each element $m$ of $M_n$ is (up to sign) the right-to-left product of the edge labels along a saturated chain that starts at the minimal element $\idelm$ and ends at $m$. See Figure~\ref{fig-mcr} for the case $M_5$.

\begin{figure}
	\begin{center}
		\begin{tikzpicture}
			\node[v] (1) at (0,0){};
			\node[v] (2) at (0,1){};
			\node[v] (3) at (0,2){};
			\node[v] (4) at (-.75, 2.75){};
			\node[v] (5) at (.75, 2.75){};
			\node[v] (6) at (-.75, 3.75){};
			\node[v] (7) at (.75, 3.75){};
			\node[v] (8) at (0, 4.5){};
			\node[v] (9) at (0, 5.5){};
			\node[v] (10) at (0,6.5){};
			\draw(1)node[left]{$\idelm$}--(2)node[left]{$\e{2}$}--(3)node[left]{$\e{12}$}--(4)node[left]{$\e{a12}$}--(6)node[left]{$\e{ba12}$}--(8)node[left]{$\e{aba12}$}--(9)node[left]{$\e{1aba12}$}--(10)node[left]{$\e{21aba12}$} (3)--(5)node[right]{$\e{b12}$}--(7)node[right]{$\e{ab12}$}--(8);
		\end{tikzpicture}
	\end{center}
	\caption{\label{fig-mcr} Hasse diagram for $M_5$, the (left) minimal coset representatives for $D_4$ inside $D_5$.}
\end{figure}
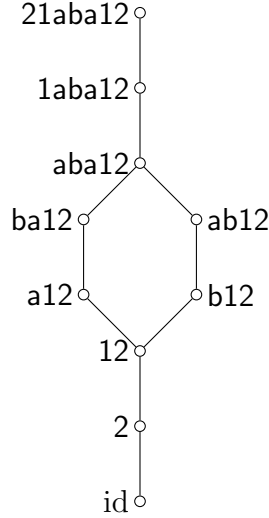

When $n=3$, $D_3=A_3$, and $\FK_{D_3}$ has basis $M_3 = \{\idelm, \e{a}, \e{b}, \e{ab}, \e{ba}, \e{aba}\}$ by Theorem~\ref{thm-a}. For $n \geq 4$, we claim that $M_n$ is a set of (left) minimal coset representatives for $\FK_{D_{n-1}}$ in $\FK_{D_{n}}$. We proceed in two steps. Let $I_n = \FK_{D_n}\FK_{D_{n-1}}^+$.

\begin{lemma} \label{lemma-d1}
The set $M_n$ spans $\FK_{D_n}/I_n$.
\end{lemma}

In other words, we can write any element of $\FK_{D_n}$ in a \emph{normal form} as a linear combination of monomials, each of which is a product of a minimal coset representative and a monomial in $\FK_{D_{n-1}}$. The proof of this lemma essentially gives a straightening algorithm for $\FK_{D_n}$.

\begin{proof}
We induct on $n$. Assume $n \geq 4$.

Consider any monomial $m \in \FK_{D_n}$. If $m$ does not contain any instance of the variable $\e{n-3}$, then it either lies in $I_n$ or is the identity element $\idelm \in M_n$.

If $m$ contains exactly one occurrence of $\e{n-3}$, then for it not to lie in $I_n$, it must equal $A\e{(n-3)}$ for some $A \in \FK_{D_{n-1}}$. By induction, $A$ is congruent modulo $I_{n-1}$ to a linear combination of elements in $M_{n-1}$. Since $\e{n-3}$ commutes with $\FK_{D_{n-2}}$, we have $I_{n-1}\cdot \e{(n-3)} \subset I_n$, so $m$ is congruent modulo $I_n$ to a linear combination of elements of $M_{n-1} \cdot \e{(n-3)} \subset M_n$.

If $m$ has more than one occurrence of $\e{n-3}$, as above we may assume that it ends with a substring of the form $\e{(n-3)}A\e{(n-3)}$, where $A \in \FK_{D_{n-1}}$. By the previous paragraph, we may assume that $A \in M_{n-1}$. If $A=\idelm$, this clearly vanishes. Otherwise, either $A=B\e{(n-4)}$ for some $B \in M_{n-2}$ or $A = \e{(n-4)\factL  aba(n-4)\factR }$. In the first case, by the braid relation
\[\e{(n-3)}A\e{(n-3)} = B\e{(n-3)(n-4)(n-3)} = B\e{(n-4)(n-3)(n-4)} \in I_n,\]
so $m$ will also lie in $I_n$.

It remains to consider the case when $m$ ends in $X=\e{(n-3)\factL aba(n-3)\factR }$. We claim that $\e{a}X=X\e{b}$, $\e{b}X=X\e{a}$, and $\e{i}X = X\e{i}$ for $\e{i} = \e{1}, \e{2}, \dots, \e{n-4}$. This shows that either $m = X \in M_n$ or $m \in I_n$, which will complete the proof.

To show $X\e{a}=\e{b}X$, since $\e{a}$ and $\e{b}$ commute with $\e{2}, \e{3}, \dots, \e{n-3}$, it suffices to show that $\e{1aba1a} = \e{b1aba1}$. This follows from the quadratic, braid, and claw relations because
\[\e{1aba1a} = \e{1ab1a1} = (\e{ab1a}+\e{b1ab})\e{a1} = \e{ab1aa1}+\e{b1aba1} = \e{b1aba1}.\]
Similarly, $X\e{b}=\e{a}X$. Finally, for $\e{i}=\e{1}, \e{2}, \dots, \e{n-4}$, since $\e{i}$ commutes with $\e{i+2}, \dots, \e{n-3}$, it suffices to show that $\e{(i+1)\factL aba(i+1)\factR i} = \e{i(i+1)\factL aba(i+1)\factR }$. This holds because
\begin{multline*}
\e{(i+1)\factL aba(i-1)\factR i(i+1)i} = \e{(i+1)\factL aba(i-1)\factR (i+1)i(i+1)} \\
= \e{(i+1)i(i+1)(i-1)\factL aba(i+1)\factR } = \e{i(i+1)i(i-1)\factL aba(i+1)\factR }.\qedhere
\end{multline*}
\end{proof}
Note that in order to put an element of $\FK_{D_n}$ into normal form, we used only the quadratic, braid, and claw relations.

\medskip

To show linear independence, we first show that the highest degree minimal coset representative is nonzero.
\begin{lemma}\label{lemma-d3}
The highest degree element $X = \e{(n-3)\factL  aba(n-3)\factR }\in M_n$ does not lie in $I_n$.
\end{lemma}
Let $K_{1, n-1}$ be the star centered at the end vertex in $D_n$, and label its edges using primed symbols as in Figure~\ref{fig-d}. (Note that, for convenience, we let $\e{(n-3)} = \e{(n-3)'}$.) Define $\e{i'\factL }$ and $\e{i'\factR }$ analogously to the unprimed versions.
\begin{proof}
By Lemma~\ref{lemma-orthogonal}, $I_n$ is orthogonal to $\FK_{K_{1, n-1}}$. Hence we need only show that $X$ is not orthogonal to some element of $\FK_{K_{1, n-1}}$. Let $X' = \e{(n-3)'\factL  a'b'a'(n-3)'\factR }$. We will show that $\langle X, X' \rangle = (-1)^{n-1}$.

For $n=3$, an easy computation shows that $\langle \e{aba}, \e{aba}\rangle = 1$. For $n\geq 4$, let $Y = \e{(n-4)\factL  aba(n-4)\factR }$ and $Y' =  \e{(n-4)'\factL  a'b'a'(n-4)'\factR }$, so that $X=\e{(n-3)}Y\e{(n-3)}$ and $X'=\e{(n-3)}Y'\e{(n-3)}$. Since
\[
\Delta_{\e{n-3}}(X') = Y'\e{(n-3)}+ \sigma_{\e{n-3}}(\e{(n-3)}Y') = Y'\e{(n-3)} - \e{(n-3)}\sigma_{\e{n-3}}(Y'),
\]
we have
\[
\langle X, X' \rangle = \langle  \e{(n-3)}Y, \Delta_{\e{n-3}}(X') \rangle = \langle\e{(n-3)}Y, Y'\e{(n-3)}\rangle - \langle \e{(n-3)}Y, \e{(n-3)}\sigma_{\e{n-3}}(Y') \rangle.
\]
But $\e{(n-3)}Y$ ends in $\e{n-4}$ (or $\e{a}$ if $n=4$), which does not occur in $Y'\e{(n-3)}$, so the first term vanishes. Thus
\[
\langle X, X' \rangle = -\langle \e{(n-3)}Y, \e{(n-3)}\sigma_{\e{n-3}}(Y') \rangle = - \langle (\e{(n-3)}Y)\nabla_{\e{n-3}}, \sigma_{\e{n-3}}(Y')\rangle.
\]
But in fact $(\e{(n-3)}Y)\nabla_{\e{n-3}} = Y$. To see this, note that none of the edges in $Y$ contain the end vertex in $D_n$, meaning $\nabla_{\e{n-3}}$ will not act on any variable in $Y$, and also that $\sigma_{Y}(\e{n-3}) = \e{n-3}$. It follows that $\langle X, X' \rangle = - \langle Y, \sigma_{\e{n-3}}(Y')\rangle$. But note that the edges appearing in $\sigma_{\e{n-3}}(Y')$ are contained in the star centered at the end vertex of $D_{n-1}$. Thus $\langle Y, \sigma_{\e{n-3}}(Y')\rangle$ is the computation for $n-1$, so the result follows by induction.
\end{proof}

\begin{lemma}\label{lemma-d2}
The set $M_n$ is linearly independent in $\FK_{D_n}/I_n$.
\end{lemma}
\begin{proof}
Note that the elements of $M_n$ all differ in either degree or $\S_n$-degree. Therefore, they are linearly independent unless one of them lies in $I_n$. But each element of $M_n$ is a right factor of the longest element $X$, so since $X$ does not lie in the left ideal $I_n$, none of the elements do.
\end{proof}

\begin{proof} [Proof of Theorem~\ref{thm-d}]
By Lemmas~\ref{lemma-d1} and \ref{lemma-d2}, $M_n$ is a set of minimal coset representatives for $\FK_{D_{n-1}}$ in $\FK_{D_{n}}$. Since the only relations needed in Lemma~\ref{lemma-d1} were the quadratic, braid, and claw relations, Lemma~\ref{lemma-d2} implies that these generate all relations in $\FK_{D_n}$.

We can now prove \[\Hilb_{D_n}(t) = [n][n-1] \cdot [4][6][8]\cdots [2n-4]\] by induction on $n$. The case $n=3$ follows from Theorem~\ref{thm-a}. For $n>3$, since $M_n$ is a set of minimal coset representatives and $\Hilb_{M_n}(t) = [n](1+t^{n-2}) = \frac{[n][2n-4]}{[n-2]}$, we have
\begin{align*}
\Hilb_{D_n}(t) &= \Hilb_{M_n}(t) \cdot \Hilb_{D_{n-1}}(t)\\
&= \frac{[n][2n-4]}{[n-2]} \cdot [n-1][n-2] \cdot [4][6] \cdots [2n-6]\\
&= [n][n-1] \cdot [4][6] \cdots [2n-4]. \qedhere
\end{align*}
\end{proof}

\subsection{The cases $\FK_{E_n}$ for $n=6$, $7$, and $8$}

In this section, we will describe $\FK_{E_n}$ for $n=6$, $7$, and $8$, where $E_n$ is the Dynkin diagram of type $E_n$ as shown in Figure~\ref{fig-e}.
\begin{figure}
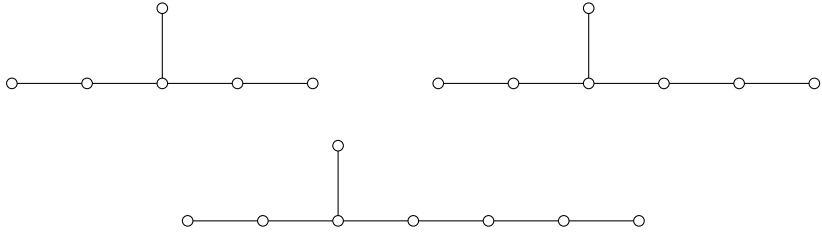

	\begin{center}
		\dr{(0,0)node[v]{}--(1,0)node[v]{}--(2,0)node[v]{}--(3,0)node[v]{}--(4,0)node[v]{} (2,0)--(2,1)node[v]{}} \quad\quad\quad
		\dr{(0,0)node[v]{}--(1,0)node[v]{}--(2,0)node[v]{}--(3,0)node[v]{}--(4,0)node[v]{}--(5,0)node[v]{} (2,0)--(2,1)node[v]{}} \\[5mm]
		\dr{(0,0)node[v]{}--(1,0)node[v]{}--(2,0)node[v]{}--(3,0)node[v]{}--(4,0)node[v]{}--(5,0)node[v]{}--(6,0)node[v]{} (2,0)--(2,1)node[v]{}}
	\end{center}
	\caption{\label{fig-e} The Dynkin diagrams $E_6$, $E_7$, and $E_8$.}
\end{figure}

\begin{thm}
The only minimal relations in $\FK_{E_6}$, $\FK_{E_7}$, and $\FK_{E_8}$ are the quadratic, braid, and claw relations. Their Hilbert series are
	\begin{align*}
	\Hilb_{E_6}(t) &= \frac{[4][5][6]^2[8][9]}{[3]},\\
	\Hilb_{E_7}(t) &= \frac{[6]^2[8][9][10][12][14]}{[3]},\\
	\Hilb_{E_8}(t) &=  \frac{[6][8][10][12][14][15][18][20][24]}{[3][5]}.
	\end{align*}
\end{thm}
Our proof of this theorem is largely computational. For this reason, we only summarize the basic strategy.
\begin{proof}
For each $n=6, 7, 8$, we first compute using \texttt{bergman} a purported set of minimal coset representatives for $\FK_{E_{n-1}}$ in $\FK_{E_n}$ (note $E_5=D_5$) assuming that the only relations are the quadratic, braid, and claw relations. This gives an upper bound on $\Hilb_{E_n}(t)/\Hilb_{E_{n-1}}(t)$.

We then use Algorithm~\ref{alg-mcr} to compute a subset of the minimal coset representatives for $\FK_{E_{n-1}}$ in $\FK_{E_n}$. This gives a lower bound on $\Hilb_{E_n}(t)/\Hilb_{E_{n-1}}(t)$.

In each case, the upper and lower bounds coincide, giving the desired Hilbert series.
\end{proof}

This method does not work for $E_9$ (also known as the affine Dynkin diagram $\tilde{E}_8$) since $\FK_{E_9}$ appears to require relations other than the quadratic, braid, and claw relations and is, in any case, beyond our current computational abilities.

\subsection{Connection to Weyl groups}

If we combine the results on $A_n$, $D_n$, and $E_n$ ($n \leq 8$) with known results about Weyl groups, we arrive at the following theorem.
\begin{thm}\label{thm-weyl}
Let $G$ be the graph of a simply-laced Dynkin diagram. Then the relations in $\FK_G$ are generated by quadratic, braid, and claw relations, and
\[\Hilb_G(t) = \frac{W_G(t)}{C_G(t)},\]
where $W_G(t)$ is the Poincar\'e series of the corresponding Weyl group $W$ and $C_G(t)$ is the characteristic polynomial of a Coxeter element in $W$. Moreover, $\dim \FK_G = |W|/f$,
where $f=C_G(1)$ is the index of connection (the index of the root lattice in the weight lattice).
\end{thm}

However, we know of no explanation of this fact that does not require explicitly computing the requisite Hilbert series first.



\section{The case $\cyc$}

In this section, we will describe $\FK_{\cyc}$, where $\cyc$ is the cycle on $n$ vertices (named after the affine Dynkin diagram). Label the edges $\e{0}, \e{1}, \dots, \e{n-1}$ in order as shown in Figure~\ref{fig-cyc}. For convenience, we will take these labels modulo $n$ so that for any integers $i$ and $j$, the edges $\e{i}$ and $\e{j}$ are equal if and only if $i \equiv j \pmod n$.

\begin{figure}
	\begin{center}
		\begin{tikzpicture}[scale = 1.5]
			\node[v] (0) at (90:1){};
			\node[v] (1) at (54:1){};
			\node[v] (2) at (18:1){};
			\node[v] (3) at (-18:1){};
			\node[v] (k-1) at (-90:1){};
			\node[v] (k) at (-126:1){};
			\node[v] (k+1) at (-162:1){};
			\node[v] (n-1) at (126:1){};
			\draw[thick, ->] (0)--node[above]{${\scriptstyle\e{0}}$}(1);
			\draw[thick, ->] (1)--node[right]{${\scriptstyle\e{1}}$}(2);
			\draw[thick, ->] (2)--node[right]{${\scriptstyle\e{2}}$}(3);
			\draw[thick, ->] (0)--node[below left]{${\scriptstyle\e{a}}$}(2);
			\draw[thick, ->] (n-1)--node[above, near start]{${\scriptstyle\e{n-1}}$}(0);
			\draw[thick, ->] (0)--node[left]{${\scriptstyle\e{b}}$}(k);
			\draw[thick, ->] (k-1)--node[below left, near start]{${\scriptstyle\e{n-k-1}}$}(k);
			\draw[thick, ->] (k)--node[below left]{${\scriptstyle\e{n-k}}$}(k+1);
			\draw[thick, loosely dotted] (-36:1) to [bend left = 18] (-72:1) (180:1) to [bend left = 18](144:1);
		\end{tikzpicture}
	\end{center}
	\caption{\label{fig-cyc} The $n$-cycle $\cyc$ with edges labeled $\e{0}, \dots, \e{n-1}$. The additional edges $\e{a}$ and $\e{b}$ will be used in the proofs of Lemma~\ref{l elem sym vanish} and Proposition~\ref{p elem sym vanish}.}
\end{figure}
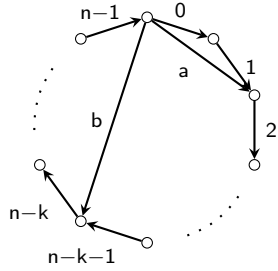

Since the line graph of $\cyc$ is isomorphic to itself, Proposition~\ref{p Coxeter to FK} shows that $\FK_\cyc$ is a quotient of the nil-Coxeter algebra $\widetilde\Nil_n$ of type $\cyc$. The main result of this section will be to describe the kernel of the quotient map explicitly as well as a set of minimal coset representatives for $\FK_{A_n}$ in $\FK_\cyc$. As a corollary, we will obtain the following result. (For a more precise statement, see Theorem~\ref{t cycle}.)

\begin{thm} \label{thm-cyc}
 The algebra $\FK_\cyc$ has a presentation consisting of quadratic relations, braid relations, and $n-1$ additional relations of degrees $k(n-k)$ for $1 \leq k \leq n-1$. The Hilbert series of $\FK_\cyc$ is given by
  \[
  \Hilb_\cyc(t)=[n]\cdot\prod^{n-1}_{k=1}[k(n-k)].
  \]
  In particular, $\dim \FK_\cyc = n! \cdot (n-1)!$, and the top degree of $\FK_\cyc$ is $\binom{n+1}{3}$.
\end{thm}

It will often be more convenient to work with an extended version of $\FK_\cyc$ defined as follows:
$\widehat{\FK}_\cyc$ is the twisted algebra $\Pi\cdot\FK_\cyc$ generated by $\Pi\cong\ZZ/ n\ZZ$ and $\FK_\cyc$, where the generator $\pi \in \Pi$ satisfies 
$\pi \e{i}=\e{(i+1)}\pi$.
Enlarging the algebra $\FK_\cyc$ to $\widehat{\FK}_\cyc$ is quite harmless---if $B$ is any basis
of $\FK_\cyc$, then  $\{\pi^i b\mid 0 \leq i \leq n-1,\; b\in B\}$ is a
basis of $\widehat{\FK}_\cyc$. We may think of $\pi$ as having degree 0 and $\S_n$-degree equal to the automorphism of $\cyc$ sending each vertex to the next adjacent vertex clockwise. Hence $\sigma_\pi(\e{i}) = \e{i+1} = \pi \e{i} \pi^{-1}$.

We begin by presenting some background on the (extended) affine symmetric group. We will then explicitly describe the relations of $\FK_\cyc$. This will allow us to find minimal coset representatives for $\FK_{A_n} = \Nil_n$ in $\FK_\cyc$ and thereby prove the main result.

\subsection{Background} Here we review some background on the geometry of the extended affine symmetric group. What follows is a somewhat cursory treatment; for a more thorough treatment of this material, see \cite{Haiman}.

The \emph{affine symmetric group} $(\widetilde\S_n, K)$ is the Coxeter group with simple reflections $K=\{s_0, s_1, \dots, s_{n-1}\}$ whose Dynkin diagram is an $n$-cycle. Here $\S_n = (\S_n, S)$ is the symmetric group generated by $S=\{s_1, \dots, s_{n-1}\}$. The \emph{extended affine symmetric group} $\eW$ is the semidirect product $\Pi \ltimes \aW$, where $\Pi$ is the cyclic group of order $n$ generated by $\pi$ satisfying $\pi s_i = s_{i+1}\pi$ (indices taken modulo $n$).

The groups $\S_n$ and $\aW$ can be realized as affine reflection groups as follows. Let $\{\epsilon_1, \dots, \epsilon_n\}$ be the standard basis of $\RR^n$ with the usual inner product $(\cdot,\cdot)$. Let $V \subset \RR^n$ be the subspace spanned by $\alpha_i = \epsilon_i-\epsilon_{i+1}$ for $i = 1, \dots, n-1$. We identify $V$ in the natural way with $\RR^n/\RR\varepsilon$, where $\varepsilon=\epsilon_1+\cdots+\epsilon_n$.

The $\alpha_i$ are the simple roots of the root system $\Phi = \Phi^+ \cup \Phi^-$ of type $A_{n-1}$, where $\Phi^+ = \{\epsilon_i-\epsilon_j \mid 1 \leq i < j \leq n\}$. For a root $\alpha \in \Phi$ and $k \in \ZZ$, denote by $h_{\alpha- k\delta}$ the (affine) hyperplane $\{x \in V \mid (x, \alpha)=k\}$, and let $s_{\alpha-k\delta}$ be the reflection over $h_{\alpha-k\delta}$. Then the map sending $s_i$ to the reflection $s_{\alpha_i}$ gives a faithful representation of $\S_n$. Denoting the highest root by $\bar\alpha  = \epsilon_1-\epsilon_n$ and sending $s_0$ to the reflection $s_{\alpha_0}$ over the hyperplane $h_{\alpha_0} = h_{-\bar\alpha+\delta}=\{x \in V \mid (x, \epsilon_1-\epsilon_n)=1\}$ extends this to a faithful representation of $\widetilde\S_n$.

The connected components of $V-\bigcup_{\alpha \in \Phi^+} h_{\alpha}$ are called \emph{chambers}, and the connected components of $V-\bigcup_{\alpha\in\Phi^+,  k\in \ZZ} h_{\alpha-k\delta}$ are called \emph{alcoves}. The actions of $\S_n$ and $\aW$ are simply transitive on chambers and alcoves, respectively. We define
\begin{align*}
\mathbf C_0 &= \{x \in V \mid (\alpha_i, x)>0, \text{ for } i=1, \dots, n-1\}\\
\mathbf A_0 &= \{x \in \mathbf C_0 \mid (\bar\alpha, x)<1\}
\end{align*}
to be the \emph{fundamental chamber} and \emph{fundamental alcove}. A point in the closure of $\mathbf C_0$ is called \emph{dominant}.

The set of all affine transformations that preserve the set of alcoves can be identified with $\eW$: let $Y$ denote the weight lattice $\ZZ^n/\ZZ\varepsilon \subset V$. Then $\eW = Y \rtimes \S_n$, where elements of $Y$ are treated as translations. In this context, $\aW = \ZZ\Phi \rtimes \S_n$, where $\ZZ\Phi \subset Y$ is the root lattice. We will denote elements $Y \subset \eW$ using the multiplicative notation $y^\lambda = y_1^{\lambda_1}\cdots y_n^{\lambda_n}$ for $\lambda = \lambda_1\epsilon_1 + \cdots + \lambda_n\epsilon_n \in Y$. Note that $y_1y_2\cdots y_n = \idelm$. Here $\Pi$ is the stabilizer of $\mathbf A_0$, and we will take $\pi = y_1s_1s_2\cdots s_{n-1}$ as a generator of $\Pi$.

Recall that if $w$ is an element of a Coxeter group, then the \emph{length} $\ell(w)$ of $w$ is the minimal length of an expression for $w$ as a product of simple reflections. For $w \in \eW$, we set $\ell(w) = \ell(v)$, where $v \in \aW$ is the unique element such that $w = \pi^k v$ for some $k$. Equivalently, $\ell(w)$ is the number of hyperplanes $h_{\alpha-k\delta}$ separating $\mathbf A_0$ from $w^{-1}(\mathbf A_0)$. 

We will sometimes abuse notation by identifying $\S_n$, $\aW$, and $\eW$ with their nil-Coxeter counterparts $\Nil_n$, $\aWnil_n$, and $\eWnil_n$. Note that a monomial $y^\lambda$, when considered as an element of $\eWnil_n$, is assumed to be a reduced expression (and therefore nonzero). We will write $\widetilde\Theta\colon \aWnil_n \to \FK_\cyc$ and $\widehat\Theta\colon \eWnil_n \to \widehat \FK_\cyc$ for the canonical surjections sending $s_i \mapsto \e{i}$. We will sometimes abuse notation by writing $\e{i}$ for $s_i \in \Nil_n$ or omitting $\widetilde\Theta$ or $\widehat\Theta$ when convenient. We will also write $w_0$ for the longest element of $\S_n$.

\subsection{Relations}
In this section, we will describe the extra relations that occur in $\FK_\cyc$.

Consider the translations $y_1, \dots, y_n \in \widehat\S_n$ as described above. Let us write each translation $y_{i_1}y_{i_2}\cdots y_{i_k} \in \widehat\S_n$ (for $1 \leq i_1<\cdots <i_k\leq n$) as a reduced word---we will see below that this has length $k(n-k)$. Let $e_k(y_1, \dots, y_n)$ be the sum of these words in $\widehat\Nil_n$.

\begin{proposition}\label{p elem sym vanish}
The image of
$e_k(y_1, \dots, y_n)$ vanishes in $\widehat{\FK}_\cyc$ for $k = 1, \dots, n-1$.
Therefore, in $\FK_\cyc$,
\[R_k = R_k(\cyc) = \pi^{-k} e_k(y_1, \dots, y_n) = 0.\]
\end{proposition}

Note that $\rev(R_k) = R_{n-k}$. Before we prove this result, we will give an explicit description of $R_k$ in terms of the generators of $\FK_{\cyc}$.

Consider a \emph{Grassmannian permutation} $w \in \S_n$ with $w_1<w_2<\dots<w_k$ and $w_{k+1}<w_{k+2}<\dots<w_n$, i.e., $w$ has at most one descent appearing at position $k$. Equivalently, $w$ is a minimal coset representative of $\S_k \times \S_{n-k}$ in $\S_n$. We associate to $w$ the partition $\lambda = (w_k-k, \dots, w_1-1)$ and write $w = \gamma(\lambda)$. This gives a bijection between such $w$ and partitions $\lambda$ with at most $k$ parts, each of size at most $n-k$.

There is a nice description of the set of reduced words of a Grassmannian permutation called the $\delta$-rule, due to Winkel \cite{Winkel}.
For a partition $\lambda$ as above, let $\delta_\lambda$ be the tableau of shape $\lambda$ with entry $k+j-i$ in the box at row $i$, column $j$.
The $\delta$-rule says that the reduced words of $\gamma(\lambda)$ are obtained by, starting with the tableau $\delta_\lambda$, successively removing outer corners and recording the entries until all the entries are removed. For example, $s_6s_2s_4s_5s_3s_4$
 is the reduced expression for $\gamma(3,2,1,0)$ corresponding to the sequence
\[
\begin{array}{cc}
{\tiny \delta_{(3,2,1,0)}\;=\;\tableau{4&5&6\\3&4\\2}\quad\to\quad \tableau{4&5\\3&4\\2} \quad\to\quad  \tableau{4&5\\3&4} \quad\to\quad \tableau{4&5\\3} \quad\to\quad \tableau{4\\3} \quad\to\quad \tableau{4}}\,.
\end{array}
\]
Given a partition  $\mu \subset \lambda$, define  $\gamma(\lambda/\mu) = \gamma(\lambda) \gamma(\mu)^{-1}$, and let $\delta_{\lambda/\mu}$ be the skew subtableau of $\delta_\lambda$ of shape $\lambda/\mu$. It follows from the $\delta$-rule that the reduced expressions for $\gamma(\lambda/\mu)$ are obtained by removing entries of $\delta_{\lambda/\mu}$ just like the $\delta$-rule.

It is easy to check from the definitions that
$y_1 \cdots y_k = \pi^{k}\cdot \gamma(\Omega)$,
where $\Omega=(n-k)^k$, the partition with $k$ parts of size $n-k$.
%
Now given any  $1 \leq i_1 < i_2 < \dots < i_k \leq n$, let $\lambda = (i_k-k, \dots, i_1-1)$. Then $y_{i_1}\cdots y_{i_k}$ has the reduced factorization
\begin{equation} \label{e y reduced}
  y_{i_1}\cdots y_{i_k} = \gamma(\lambda) y_1 \cdots y_k \gamma(\lambda)^{-1} = \gamma(\lambda) \cdot \pi^k \cdot \gamma(\Omega/\lambda). \tag{$\spadesuit$}
\end{equation}
Therefore
\[
\pi^{-k} y_{i_1}\cdots y_{i_k} = \pi^{-k} \gamma(\lambda) \pi^k \cdot \gamma(\Omega/\lambda) = \sigma_{\pi}^{-k}(\gamma(\lambda)) \cdot \gamma(\Omega/\lambda).
\]
Summing over all $\lambda \subset \Omega$ gives $R_k$.
\begin{ex}
If $n=5$, then
\begin{align*}
R_1 &= \e{4321} + \e{0\cdot 432} + \e{10 \cdot 43} + \e{210 \cdot 4} + \e{3210},\\
R_2 &= \e{342312} + \e{0 \cdot 34231} + \e{10 \cdot 3421} + \e{40 \cdot 3423} + \e{210 \cdot 321}\\
& \qquad + \e{410 \cdot 342} + \e{4210 \cdot 32} + \e{0410 \cdot 34} + \e{04210 \cdot 3} + \e{104210},\\
R_3 &= \e{234123} + \e{0 \cdot 23412} + \e{10 \cdot 2312} + \e{40 \cdot 2341} + \e{410 \cdot 231}\\
& \qquad + \e{340 \cdot 234} + \e{0410 \cdot 21} + \e{3410 \cdot 23} + \e{30410 \cdot 2} + \e{430410},\\
R_4 &= \e{1234}+ \e{0 \cdot 123} + \e{40 \cdot 12} + \e{340 \cdot 1} + \e{2340}.
\end{align*}
See also Example~\ref{ex-a3tilde} for the case $n=4$.
\end{ex}

We now prove the following special case of Proposition~\ref{p elem sym vanish}.
\begin{lemma} \label{l elem sym vanish}
The following relations hold in $\FK_{\cyc}$:
\begin{align*}
R_1&= \sum_{i=0}^{n-1} \e{(i-1)(i-2)\cdots 0(n-1)(n-2)\cdots(i+1)}=0,\\
R_{n-1}&= \sum_{i=0}^{n-1} \e{(i+1)(i+2)\cdots(n-1)01\cdots(i-1)}=0.
\end{align*}
\end{lemma}
\begin{proof}
We prove the $R_1$ relation, the $R_{n-1}$ relation being similar.

We induct on $n$. For the base case $n=3$, $R_1 = \e{21} + \e{02} + \e{10}$ is one of the quadratic relations of $\FK_{\tilde{A}_2}$. For $n>3$, consider the additional edge $\e{a}$ as drawn in Figure~\ref{fig-cyc}, and let $C$ be the cycle $\e{2}, \e{3}, \dots, \e{n-1}, \e{a}$. For $i=2, 3, \dots, n-1$, let \[X_i = \e{(i-1)(i-2)\cdots 2\cdot a\cdot (n-1)(n-2)\cdots (i+1)}.\] Then by induction, $R_1(C) = \e{(n-1)(n-2) \cdots 2} + \sum_{i=2}^{n-1} X_i=0$.
Using the commutation relations and the quadratic relation $\e{0a} + \e{a1} = \e{10}$,
\[\e{0}\cdot X_i + X_i \cdot \e{1} = \e{(i-1)\cdots 2}\cdot (\e{0a}+\e{a1}) \cdot \e{(n-1)\cdots (i+1)} = \e{(i-1)\cdots 210(n-1) \cdots (i+1)}.\]
It now follows easily that $R_1(\cyc)=\e{0} \cdot R_1(C) + R_1(C) \cdot \e{1}=0$, as desired.
\end{proof}

Using Lemma~\ref{l elem sym vanish}, we can complete the proof of Proposition~\ref{p elem sym vanish}.

\begin{proof}[Proof of Proposition \ref{p elem sym vanish}]
Lemma \ref{l elem sym vanish} establishes the cases $k=1$ and $k= n-1$.  Suppose that $n>3$ and $k \in \{2,\dots, n-2\}$.
To prove that $e_k(y_1,\dots, y_n)$ is zero in $\widehat{\FK}_\cyc$, write
\[
  e_k(y_1,\dots, y_n) = e_k(y_1, \dots, y_{n-1}) + e_{k-1}(y_1, \dots, y_{n-1})y_n
\]
We will show that
\begin{equation}\label{e ek split}
   - e_k(y_1, \dots, y_{n-1}) = e_{k-1}(y_1, \dots, y_{n-1})y_n. \tag{$*$}
\end{equation}
in $\widehat{\FK}_{\cyc \cup \{\e{b}\}}$, where $\e{b}$ is the additional edge shown in Figure~\ref{fig-cyc}.

Maintain the notation of \eqref{e y reduced} with $\Omega = (n-k)^k$ and $\Omega^L = (n-k-1)^k$. The left side of \eqref{e ek split} contains terms $y_{i_1}\cdots y_{i_k}$ for which $i_k \neq n$. Then $\lambda_1 = i_k-k < n-k$, so $\gamma(\Omega/\lambda) = \e{(n-k)(n-k+1)\cdots (n-1)} \cdot \gamma(\Omega^L/\lambda)$ is a reduced factorization (corresponding to first removing the rightmost column of $\delta_{\Omega/\lambda}$), which yields
\begin{equation} \label{e y reduced ik less n}
-y_{i_1} \cdots y_{i_k} = -\gamma(\lambda) \pi^k \e{(n-k)(n-k+1)\cdots (n-1)} \gamma(\Omega^L/\lambda).\tag{$**$}
\end{equation}
Using the $R_k$ relation of the cycle $\e{n-k}, \dots, \e{n-1}, \e{b}$, which is
\[
-\e{(n-k)\cdots (n-1)} = \sum_{i=1}^k \e{(n-i+1)\cdots (n-1) \cdot  b \cdot  (n-k) \cdots (n-i-1)},
\]
we can expand the right side of \eqref{e y reduced ik less n} into $k$ monomials.
The left side of \eqref{e ek split} then expands into $k\binom{n-1}{k}$ monomials of the form
\begin{multline*}
\gamma(\lambda) \pi^k \e{(n-i+1) \cdots (n-1) \cdot  b \cdot   (n-k) \cdots (n-i-1)} \gamma(\Omega^L/\lambda)\\
=\gamma(\lambda) \e{(k-i+1) \cdots (k-1)}\cdot \pi^k  \e{ b \cdot (n-k) \cdots (n-i-1)} \gamma(\Omega^L/\lambda),
\end{multline*}
where $\lambda \subset \Omega^L$ and $1 \leq i \leq k$. But $\gamma(\lambda)\e{(k-i+1)\cdots (k-1)}$ ranges over all minimal coset representatives $\S_{n-1}^J$ for the parabolic subgroup $\S_{k-1} \times \S_1 \times \S_{n-k-1}$ (generated by $J=\{s_1, \dots, s_{k-2}, s_{k+1}, \dots, s_{n-2}\}$) in $\S_{n-1}$: indeed, $\gamma(\lambda)$ are the minimal coset representatives for $\S_k \times \S_{n-k-1}$ in $\S_{n-1}$, while $\e{(k-i+1)\cdots (k-1)}$ are the ones for $\S_{k-1} \times \S_1$ in $\S_k$. A similar argument shows that each $\e{(n-k)\cdots (n-i-1)}\gamma(\Omega^L/\lambda)$ is a reduced expression. Hence
\[-e_k(y_1, \dots, y_{n-1}) = \sum_{w \in \S_{n-1}^J} w\pi^k\e{b}w',\]
where $w'$ is the unique element of $\S_n$ such that the $\S_n$-degree of $w\pi^k \e{b}w'$ is the identity.

Similarly, the right side of \eqref{e ek split} contains terms $y_{i_1}\cdots y_{i_k}$ for which $i_k = n$ so that $\lambda_1 = n-k$. Then $\lambda^B=(\lambda_2, \dots, \lambda_k) \subset \Omega^B = (n-k)^{k-1}$, so we find (by removing the top row of $\delta_\lambda$ last) that $\gamma(\lambda) = \gamma(\lambda^B) \cdot \e{(n-1)(n-2) \cdots k}$. Thus
\begin{align*}
y_{i_1}\cdots y_{i_k} &= \gamma(\lambda^B)\e{(n-1)(n-2)\cdots k}\pi^k \gamma(\Omega^B/\lambda^B)\\
&=  \gamma(\lambda^B)\pi^k\e{(n-k-1)(n-k-2)\cdots 0} \gamma(\Omega^B/\lambda^B).
\end{align*}
Then using the $R_1$ relation of the cycle $\e{0}, \e{1}, \dots, \e{n-k-1}, \e{b}$, the right side of \eqref{e ek split} expands into $(n-k)\binom{n-1}{k-1}$ monomials of the form
\begin{multline*}
\gamma(\lambda^B)\pi^k \e{(i-1) \cdots 0\cdot b \cdot (n-k-1) \cdots (i+1)}\gamma(\Omega^B/\lambda^B) \\
=  \gamma(\lambda^B) \e{(k+i-1) \cdots k} \cdot \pi^k  \e{b \cdot(n-k-1) \cdots (i+1)}\gamma(\Omega^B/\lambda^B)
\end{multline*}
for $\lambda^B \subset \Omega^B$ and $0 \leq i \leq n-k-1$. But as above, $\gamma(\lambda^B) \e{(k+i-1) \cdots k}$ ranges over $\S_{n-1}^J$ since $\gamma(\lambda^B)$ are the minimal coset representatives for $\S_{k-1} \times \S_{n-k}$ in $\S_{n-1}$ while $\e{(k+i-1) \cdots k}$ are the ones for $\S_1 \times \S_{n-k-1}$ in $\S_{n-k}$. Thus as above we have 
\[e_{k-1}(y_1, \dots, y_n)y_n = \sum_{w \in \S_{n-1}^J} w\pi^k\e{b}w' = -e_k(y_1, \dots, y_{n-1}).\qedhere\]
\end{proof}

\begin{example}
We illustrate the proof of Proposition \ref{p elem sym vanish} in the case $n = 5$, $k=2$.  There, using the relation $\e{210} = \e{10b} + \e{0b2} + \e{b21}$:
\begin{align*}
\pi^{-2} e_1(y_1, y_2, y_3, y_4)y_5 \quad = & \phantom{{}+{}}  \e{210321} + \e{421032} + \e{042103} + \e{104210}\\[1.5mm]
   =&\phantom{{}+{}}\e{10b321 + 410b32 + 0410b3 + 10410b}\\
  &+\e{0b2321 + 40b232 + 040b23 + 1040b2}\\
  &+\e{b21321 + 4b2132 + 04b213 + 104b21},
\end{align*}
Similarly, using the relation $\e{-34} = \e{4b} + \e{b3}$:
\begin{align*}
-\pi^{-2}e_2(y_1, y_2, y_3, y_4) \quad = & -   \e{342312 - 034231 - 103421 - 403423 - 410342 - 041034}\\[1.5mm]
  =&\phantom{{}+{}}\e{4b2312 + 04b231 + 104b21 + 404b23 + 1404b2 + 04104b}\\
    &+ \e{b32312 + 0b3231 + 10b321 + 40b323 + 140b32 + 0410b3}\\[1.5mm]
  =&\phantom{{}+{}}\e{4b2132 + 04b213 + 104b21 + 040b23 + 1040b2 + 10410b}\\
   &+  \e{b21321 + 0b2321 + 10b321 + 40b232 + 410b32 + 0410b3}.
\end{align*}

  The last equality uses only braid and commutation relations.  We see directly in this case that the twelve terms that appear in each case are the same.
\end{example}

\subsection{Primitive elements} Let $I_0 \subset \widehat\Nil_n$ be the (two-sided) ideal generated by the relations $R_k$, or equivalently by $e_k(y_1, \dots, y_n)$. By Proposition~\ref{p elem sym vanish}, $\widehat\FK_\cyc$ is a quotient of $\widehat \Nil_n/I_0$. To show that they are isomorphic, we first describe a basis of $\widehat \Nil_n/I_0$ in terms of a subset of $\widehat\S_n$ which we call \emph{primitive elements}, previously studied in \cite{LJantzen, Xi, SW, BlasiakFactor, BlasiakCyclage}.  These primitive elements will turn out to form a set of minimal coset representatives of $ \Nil_n = \FK_{A_n}$ in $\widehat\FK_\cyc$. We will work with a geometric description of primitive elements from \cite{LJantzen}.

Let a \emph{box} be a connected component of $H - \bigcup_{i \in [n-1], k \in \ZZ} h_{\alpha_i - k\delta}$. We denote by $\mathbf{B}_0$ the box containing $\mathbf{A}_0$, which lies between the hyperplanes $h_{\alpha_i}$ and $h_{\alpha_i - \delta}$ for $i \in [n-1]$. An element $w$ in $\eW$ is \emph{primitive} if $w^{-1}(\mathbf{A}_0) \subseteq \mathbf{B}_0$. Let $D^S$ denote the set of primitive elements of  $\eW$. Note that since $\mathbf B_0$ is contained in the fundamental chamber $\mathbf C_0$, every primitive element is an element of $\widehat\S_n^S$, that is, a (left) minimal coset representative of $\S_n$ in $\widehat\S_n$.

The elements of $D^S$ are in bijection with elements of $\S_n$: for any $w \in \S_n$, there is a unique $y^\lambda$ such that $y^\lambda w \in D^S$. Also, since $\pi$ stabilizes $\mathbf A_0$, $v \in D^S$ if and only if $\pi v \in D^S$.

\begin{example}\label{ex SL3}
The primitive elements of $\widehat{\S}_3$, expressed as products of simple reflections (top line)
and using the fact that $\widehat{S}_3 = Y\rtimes \S_3$ (bottom line), are:
\[
{
\setlength\arraycolsep{15pt}
\begin{array}{cccccc}
\idelm & \pi & \pi^2 & \pi s_0& \pi^2 s_0 & \pi^3 s_0\\
\idelm&y_1s_1s_2&y_1y_2s_2s_1&y_2s_2&y_1y_3s_1&y_1^2y_2s_1s_2s_1
\end{array}
}
\]
\end{example}

\subsection{Spanning}
We will now show how to construct a spanning set of $\widehat \Nil_n/I_0$ from the primitive elements. We will see in Theorem~\ref{t cycle} that it is actually a basis. We will first need the following lemma about reduced factorizations.

\begin{lemma} \label{lemma-redfac}
Suppose $\lambda \in Y$ is a dominant weight and $w \in \widehat\S_n^S$. Then $y^\lambda \cdot w^{-1}$ is a reduced factorization.
\end{lemma}
\begin{proof}
It suffices to show that no hyperplane that separates $y^{-\lambda}(\mathbf A_0)$ and $\mathbf A_0$ also separates $\mathbf A_0$ and $w^{-1}(\mathbf A_0)$. Since $(-\lambda, \alpha) \leq 0$ for any $\alpha \in \Phi^+$, any hyperplane separating $y^{-\lambda}(\mathbf A_0)$ from $\mathbf A_0$ has the form $h_{\alpha-k\delta}$, where $\alpha\in \Phi^+$ and $k \leq 0$. Likewise, $w^{-1}(\mathbf A_0)$ lies in $\mathbf C_0$, so any hyperplane separating $\mathbf A_0$ and $w^{-1}(\mathbf A_0)$ has the form $h_{\alpha-k\delta}$ for some $\alpha \in \Phi^+$ and $k>0$.
\end{proof}

Let $\Nil_n^d$ be the degree $d$ part of $\Nil_n$, and let $\Nil_n^+ = \bigoplus_{d>0} \Nil_n^d$.

\begin{lemma} \label{lemma-primitive span}
If $x \in \widehat\S_n^S$ but $x \not\in D^S$, then $x \in I_0+\widehat\Nil_n\Nil_n^+$.
\end{lemma}
\begin{proof}
By our choice of $x$, $x^{-1}(\mathbf A_0)$ lies in a box $\mathbf B = y^\lambda(\mathbf B_0) \neq \mathbf B_0$ contained in the fundamental chamber. Hence by Lemma~\ref{lemma-redfac}, $x^{-1}$ has a reduced factorization of the form $y^\lambda\cdot v^{-1}$, where $\lambda$ is a nonzero dominant weight and $v \in D^S$. Thus it suffices to show that the image of $y^{\lambda}$ lies in $I_0 + \Nil_n^+\widehat\Nil_n$.

Choose any $k$ such that $\lambda_k>\lambda_{k+1}$. The only term of $e_k(y_1, \dots, y_n)$ that corresponds to a dominant weight is $y^\mu = y_1\cdots y_k$; hence all other terms lie in $\Nil_n^+\widehat\Nil_n$ and so $y^\mu \in I_0 + \Nil_n^+\widehat\Nil_n$. By Lemma~\ref{lemma-redfac}, $y^{\lambda} = y^\mu \cdot y^{\lambda-\mu}$ is a reduced factorization, so the result follows.
\end{proof}

\begin{prop} \label{prop-primitive span}
The image of $\{vw \mid v \in D^S, w \in \Nil_n\}$ spans $\widehat\Nil_n/I_0$.
\end{prop}
\begin{proof}
Since $\{xw \mid x \in \widehat\S_n^S, w \in \Nil_n\}$ is a basis for $\widehat\Nil_n$, it spans $\widehat\Nil_n/I_0$. Suppose $w \in \Nil_n^{d}$. By Lemma~\ref{lemma-primitive span}, if $x \not \in D^S$, then $xw \in I_0 + \widehat\Nil_n\Nil_n^+w \subset I_0 + \widehat\Nil_n\Nil_n^{d+1}$. Hence $\widehat\Nil_n\Nil_n^d / \widehat\Nil_n\Nil_n^{d+1}$ is spanned by $\{vw \mid v \in D^S, w \in \Nil_n^d\}$ and elements of $I_0$. (In particular, $\widehat\Nil_n\Nil_n^{\ell(w_0)}$ is spanned by $\{vw_0 \mid v \in D^S\}$.) The desired result then follows from a straightforward induction on $\ell(w_0)-d$.
\end{proof}

\subsection{Linear independence}
In this section, we will prove that the images of the primitive elements under $\widehat\Theta$ are linearly independent as (left) minimal coset representatives of $\FK_{A_n} = \Nil_n$ in $\widehat \FK_\cyc$ by computing appropriate evaluations of the bilinear form $\langle \cdot, \cdot \rangle$.

First we recall some facts about Bruhat order in $\widetilde\S_n$. A product of simple reflections $s_{i_1}s_{i_2}\cdots s_{i_d}$ can be visualized by the alcove walk (see, e.g., \cite{Ram})
\[
\mathbf D_0 = \mathbf{A}_0,\quad \mathbf D_1 = s_{i_1}(\mathbf{A}_0),\quad \mathbf D_2 = s_{i_1}s_{i_2}(\mathbf{A}_0),\quad\dots,\quad \mathbf D_d = s_{i_1}s_{i_2}\cdots s_{i_d}(\mathbf{A}_0).
\]
Here $\mathbf D_j$ is the reflection of $\mathbf D_{j-1}$ across one of its facets, namely the one contained in the hyperplane $h_{\beta^j}= s_{i_1}\cdots s_{i_{j-1}}(h_{\alpha_{i_j}})$.

Suppose we omit a simple reflection $s_{i_j}$ from the product $s_{i_1}\cdots s_{i_d}$, giving the product 
$s_{i_1}\cdots \widehat{s_{i_j}}\cdots s_{i_d}$. Then the resulting alcove walk is obtained from that
of $s_{i_1}\cdots s_{i_d}$ by reflecting the second part of the walk across $h_{\beta^j}$. More
precisely, its alcove walk is
\[
\mathbf{D}_0, \dots,\,\mathbf{D}_{j-1},\, s_{\beta^j}(\mathbf{D}_{j+1}),\,\dots,\, s_{\beta^j}(\mathbf{D}_d).
\]

It is well known that $s_{i_1}\cdots s_{i_d}$ is a reduced expression if and only if its alcove walk
never crosses a given hyperplane more than once. We claim that
\begin{equation}\label{etext crossed twice}
\parbox{14cm}{
If $s_{i_1}\cdots s_{i_d}$ and $s_{i_1}\cdots \widehat{s_{i_j}}\cdots s_{i_d}$
are reduced expressions, then $h_{\beta^j}$ is either the first or last hyperplane parallel to $h_{\beta^j}$ crossed in the alcove walk of $s_{i_1}\cdots s_{i_d}$.
} \tag{$\diamondsuit$}
\end{equation}
If not, then the alcove walk of $s_{i_1}\cdots s_{i_d}$ crosses $h_{\beta^j+\delta}$ and $h_{\beta^j-\delta}$. But then the alcove walk of $s_{i_1}\cdots \widehat{s_{i_j}}\cdots s_{i_d}$ crosses one of them twice (since $s_{\beta^j}(h_{\beta^j+\delta}) = h_{\beta^j-\delta}$).

\begin{lemma}\label{l rho pairing}
Suppose $w = p_1p_2\cdots p_d \in \widetilde\S_n$ is a reduced expression with $p_1\cdots p_{\ell(w_0)} = w_0$ and $w^{-1}w_0 = p_dp_{d-1} \cdots p_{\ell(w_0)+1}$ primitive. If $j \in [d]$ is an index such that $p_1\cdots \widehat{p_j} \cdots p_d$ is a reduced expression and $p_j = \sigma_{p_{j+1}\cdots p_d}(p_d)$, then $j=d$.
\end{lemma}
\begin{proof}
Since $p_j = \sigma_{p_{j+1}\cdots p_d}(p_d)$, it follows that $p_j \cdots p_{d-1}$ and $p_{j+1} \cdots p_{d}$ have the same $\S_n$-degree. Hence $p_1 \cdots \widehat{p_j} \cdots p_d$ and $p_1 \cdots p_{d-1}$ have the same $\S_n$-degree, so they differ by a translation in $\widetilde \S_n$. Defining $\beta^j$ as in the discussion above, we have that $p_1 \cdots \widehat{p_j} \cdots p_d = s_{\beta^j}w$ and $p_1 \cdots p_{d-1} = s_{\beta^d}w$, so the hyperplanes $h_{\beta^j}$ and $h_{\beta^d}$ are parallel.

By \eqref{etext crossed twice} and the assumptions of the lemma, $h_{\beta^j}$ must be the first
or last hyperplane parallel to $h_{\beta^j}$ crossed in the alcove walk of $p_1 p_2\cdots p_d$. If
it is the last, then $d=j$ and we are done, so assume it is the first.
Since $p_1\cdots p_{\ell(w_0)}$ is a reduced expression for $w_0$, the first $\ell(w_0)$ hyperplanes
 crossed in the alcove walk of $p_1\cdots p_d$ contain all $\ell(w_0)$ different hyperplane directions.
 Therefore, $j\leq \ell(w_0)$. Set $y=p_1\cdots \widehat{p_j}\cdots p_{\ell(w_0)}$, which
 is a reduced expression for $y$. Therefore $\ell(y)=\ell(w_0)-1$, so it follows that
 $y= s_cw_0$ for some simple reflection $s_c \in \S_n$. Hence $s_{\beta^j}w = p_1 \cdots \widehat{p_j} \cdots p_d = s_cw$, so $\beta^j = \alpha_c$.

Since $w^{-1}w_0$ is primitive, $w_0w(\mathbf A_0) \subset \mathbf B_0$, or equivalently $w(\mathbf A_0) \subset w_0(\mathbf B_0)$. The region $w_0(\mathbf{B}_0)$ is also a box and lies between $h_{\alpha_i}$ and $h_{\alpha_i+\delta}$ for $i\in [n-1]$. We conclude that the only hyperplane separating $\mathbf{A}_0$ and $w(\mathbf{A}_0)$ and parallel to $h_{\alpha_c}$ is $h_{\alpha_c}$ itself. Hence $j=d$, as desired.
\end{proof}

We can now prove the following result along the same lines as the proof of Theorem~\ref{thm-a}. For $w, w' \in \widetilde \Nil_n$, write $\langle w, w' \rangle = \langle \widetilde\Theta(w), \widetilde\Theta(w') \rangle$.
\begin{prop}\label{p primitive pairing}
If  $v \in \aWnil_n$ is primitive, then $\langle w_0\rev (v), vw_0\rangle=1$.
\end{prop}
\begin{proof}
First note that if $p_1\cdots p_d$ is not a reduced expression in $\aW$, then $\langle p_1\cdots p_d, Q\rangle=0$ for all $Q\in \FK_n$.  (This is because $p_1\cdots p_d = 0$ in $\widetilde\Nil_n$ and hence also in $\FK_\cyc$.)

Let $p_1\cdots p_d=w_0\rev(v)$ as in Lemma \ref{l rho pairing}. If $v = \idelm$, then $\langle w_0, w_0 \rangle = 1$ by the proof of Theorem~\ref{thm-a}. Otherwise, by Lemma \ref{lemma-leibniz}, Proposition \ref{prop-pairing}, and Lemma \ref{l rho pairing},
\begin{align*}
\langle p_1\cdots p_d,\;p_d \cdots p_1\rangle
 &=\sum_{j=1}^d (p_j) (\sigma_{p_{j+1} \cdots p_d}\nabla_{p_d}) \cdot \langle p_1 \cdots \widehat{p_j} \cdots p_d,\;
 p_{d-1}p_{d-2}\cdots p_1 \rangle\\
 &=\langle p_1\cdots p_{d-1},\; p_{d-1} \cdots p_1\rangle.
 \end{align*}
But $p_1 \cdots p_{d-1} = w_0\rev (v')$ for some primitive element $v'$: since $v$ is primitive, the alcove walk of $\rev(v)$ does not leave $\mathbf B_0$, so neither does that of $\rev(v')$. The result then follows by induction.
\end{proof}

It is now straightforward to prove that the primitive elements satisfy an appropriate linear independence property.

\begin{prop} \label{prop-primitive ind}
The images of the primitive elements $D^S \subset \widehat\Nil_n$ under $\widehat\Theta$ form a subset of the (left) minimal coset representatives of $\Nil_n = \FK_{A_n}$ in $\widehat\FK_\cyc$, i.e., they are linearly independent modulo $\widehat\FK_\cyc \FK_{A_n}^+$.
\end{prop}
\begin{proof}
We need to show that if $\sum_{v \in D^S} c_v\widehat\Theta(v) \in \widehat\FK_\cyc \FK_{A_n}^+$ for some $c_v \in \QQ$, then all the $c_v$ vanish. We may assume that all $v$ lie in $\widetilde\Nil_n$. Multiplying the above expression by $w_0$, we have that
\begin{equation}\label{e linear ind FK}
\sum_{ v\in D^S} c_{v} \widehat{\Theta} (v w_0) =0. \tag{$*$}
\end{equation}
Since all elements of $D^S$ have different $\S_n$-degree, $\langle w_0 \rev(v'), vw_0 \rangle=0$ unless $v=v'$, in which case it equals 1 by Lemma~\ref{l rho pairing}. Thus pairing \eqref{e linear ind FK} with $\widehat\Theta(w_0 \rev(v))$ gives $c_v=0$.
\end{proof}


\subsection{Proof of main theorem}
Combining the results of the previous sections, we can now prove Theorem~\ref{thm-cyc}. It is an immediate consequence of the following more explicit result.

\begin{thm} \label{t cycle}
The algebra $\widehat\FK_\cyc$ is isomorphic to $\widehat\Nil_n/I_0$, where $I_0$ is the ideal generated by $e_k(y_1, \dots, y_n)$ for $k=1, \dots, n-1$. Moreover, $D^S$ is a set of (left) minimal coset representatives for $\Nil_n = \FK_{A_{n}}$ in $\widehat\FK_\cyc$. 
\end{thm} 
\begin{proof}
By Proposition~\ref{p elem sym vanish}, $\widehat\FK_\cyc$ is a quotient of $\widehat\Nil_n/I_0$. By Proposition~\ref{prop-primitive span}, $\{vw \mid v \in D^S, w \in \Nil_n\}$ spans $\widehat\Nil_n/I_0$, and by Proposition~\ref{prop-primitive ind} and Theorem~\ref{thm-subgraph}, its image is linearly independent in $\widehat\FK_\cyc$. It follows that it must be a basis of both.
\end{proof}

\begin{cor}
 The Hilbert series of $\FK_\cyc$ is given by
  \[
  \mathcal{H}_\cyc(t)=[n]\cdot\prod^{n-1}_{i=1}[i(n-i)].
  \]
\end{cor}
\begin{proof}
By Remark 1.5 and Lemma 2.2 of \cite{SW}, the length generating function for the subset
$D^S\subset \eWnil_n$ is given by
\[
\sum_{v\in D^S }t^{\ell(v)}= n\cdot \prod_{i=1}^{n-1}\frac{1-t^{i(n-i)}}{1-t^{i}}.
\]
Hence by Theorem \ref{t cycle}, the Hilbert series of $\widehat{\FK}_\cyc$ is given by
\[
\sum_{\substack{v \in D^S\\ w\in\fWnil_n}} t^{\ell(v w)}
= [n]!\cdot n\cdot \prod_{i=1}^{n-1}\frac{1-t^{i(n-i)}}{1-t^i}
=n\cdot\prod_{i=1}^{n}\frac{1-t^i}{1-t}\cdot \prod_{i=1}^{n-1}\frac{1-t^{i(n-i)}}{1-t^i}
=n\cdot [n]\cdot\prod_{i=1}^{n-1}[i(n-i)].
\]
Since the Hilbert series of $\widehat{\FK}_\cyc$ is equal to $n$ times the
Hilbert series of $\FK_\cyc$, the desired result follows.
\end{proof}

\begin{rmk}\label{r coinvariants}
Let $I$ be the ideal of $\QQ[y_1, \dots, y_n]$ generated by symmetric polynomials of positive degree.
The quotient of $\QQ [y_1, y_2, \dots, y_n]$
by $I$ is called the \emph{ring of coinvariants} and is known to have dimension $n!$.

The reference \cite{BlasiakCyclage} identifies a  subalgebra $\pH_n$ of the extended affine Hecke algebra $\eH_n$ of type $A$.
The subalgebra $\pH_n$ is a $q$-analogue of $\QQ\S_n \ltimes \QQ [y_1,\dots,y_n]$
and inherits a canonical basis from that of $\eH_n$.
Let $\mathcal{I}$ denote the two-sided ideal of $\pH_n$ generated by the symmetric polynomials of
positive degree in the Bernstein generators.
Corollary $6.7$ of \cite{BlasiakCyclage} states that the $\pH_n$-module $\pH_n \C_{w_0}/ \mathcal{I}\C_{w_0}$
has canonical basis $\{\C_{vw_0} \mid v\in D^S\}$, where $\C_w$ denotes the canonical basis element labeled by
$w\in\eW$. The basis $\{vw_0 \mid v\in D^S\}$ of $\eWnil_n w_0/I_0 w_0$ is essentially
the $q=0$ specialization of this canonical basis. (It is not exactly the $q=0$ specialization because the
canonical basis of \cite{BlasiakCyclage} is for the $G=GL_n$ extended version of the affine Weyl group,
whereas here we have used the $G=SL_n$ version.)
\end{rmk}


\section{Questions and conjectures}

We conclude with some questions and conjectures to guide further research.

\begin{question}
What is the Hilbert series $\Hilb_G(t)$ of $\FK_G$?
\end{question}
Towards this end, we have a conjecture for the Hilbert series of the affine Dynkin diagrams $\tilde{D}_n$ as shown in Figure~\ref{fig-dntilde}. This conjecture was obtained by using Algorithm~\ref{alg-mcr} to compute minimal coset representatives for $\FK_{D_n}$ in $\FK_{\tilde D_n}$. Here, $[n]!! = [n][n-2][n-4]\cdots$.

\begin{figure}
	\begin{center}
		\begin{tikzpicture}
			\draw (150:1)node[v]{}--(0,0)node[v]{}--(1,0)node[v]{}--(1.5,0) (-150:1)node[v]{}--(0,0);
			\draw[thick, loosely dotted] (1.5,0)--(2.5,0);
			\draw (2.5,0)--(3,0)node[v]{}--(4,0)node[v]{}--+(30:1)node[v]{} (4,0)--+(-30:1)node[v]{};
		\end{tikzpicture}
	\end{center}
\caption{\label{fig-dntilde}The affine Dynkin diagram $\tilde D_n$ with $n+1$ vertices.}
\end{figure}
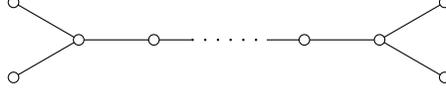

\begin{conj} \label{conj-dtilde}
The Hilbert series for $\FK_{\tilde{D}_n}$ is
\[
\frac{[2n-2]!!}{[2n-3]!!} \cdot \frac{[n][n+1]}{[2]^2[n-1][n-2]} \cdot \left[\frac{n^2-n}{2}\right]^2 \cdot \left[\frac{n^2-n-2}{2}\right]^2 \cdot \prod_{i=1}^{n-3} [i(2n-i-1)].
\]
In particular, $\dim \FK_{\tilde D_n} = (n+1)^2\cdot 2^{2n-8}$, and the top degree of $\FK_{\tilde D_n}$ is $\frac13n(n-1)(2n-1)$.
\end{conj}
For small values of $n$ (including $n=3$, for which $\tilde{D}_3=\tilde{A}_3$), this gives the following:
\begin{align*}
n=3&&&[3]^2[4]^2\\
n=4&&&\frac{[4]^2[5]^2[6]^4}{[2]^2[3]^2}\\
n=5&&&\frac{[6]^2[8]^2[9]^2[10]^2[14]}{[2][3]^2[7]}\\
n=6&&&\frac{[6]^2[8][10]^2[14]^2[15]^2[18][24]}{[2][3][5]^2[9]}\\
n=7&&&\frac{[4][8]^2[10][12]^2[20]^2[21]^2[22][30][36]}{[2][3][5]^2[9][11]}
\end{align*}

\begin{figure}
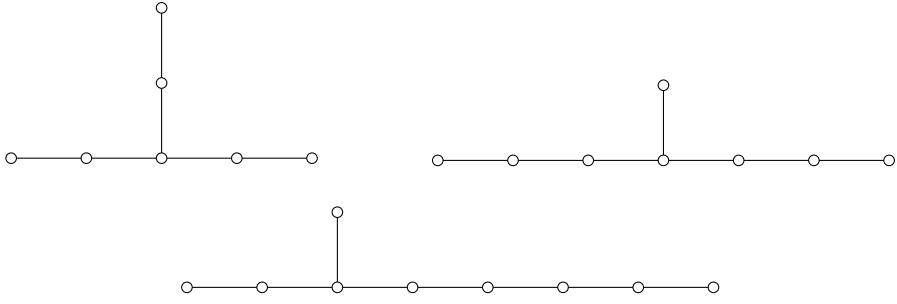

	\begin{center}
		\dr{(-2,0)node[v]{}--(-1,0)node[v]{}--(0,0)node[v]{}--(1,0)node[v]{}--(2,0)node[v]{} (0,0)--(0,1)node[v]{}--(0,2)node[v]{}}
		\qquad\quad
		\dr{(-3,0)node[v]{}--(-2,0)node[v]{}--(-1,0)node[v]{}--(0,0)node[v]{}--(1,0)node[v]{}--(2,0)node[v]{}--(3,0)node[v]{} (0,0)--(0,1)node[v]{} (0,2)node{}}\\[5mm]
		\dr{(-2,0)node[v]{}--(-1,0)node[v]{}--(0,0)node[v]{}--(1,0)node[v]{}--(2,0)node[v]{}--(3,0)node[v]{}	--(4,0)node[v]{}--(5,0)node[v]{} (0,0)--(0,1)node[v]{}}	
	\end{center}
\caption{\label{fig-entilde}The affine Dynkin diagrams $\tilde E_6$, $\tilde E_7$, and $\tilde E_8$.}
\end{figure}

We also have conjectures regarding the affine Dynkin diagrams $\tilde{E}_6$ and $\tilde{E}_7$, as shown in Figure~\ref{fig-entilde}.
\begin{conj} \label{conj-etilde}
The Hilbert series for $\FK_{\tilde{E}_6}$ and $\FK_{\tilde{E}_7}$ are:
\begin{align*}
\Hilb_{\tilde{E}_6}(t) &= \frac{[6][9][12][14]^2[16]^2[21][22][30]^2}{[3]^2[4][7][11]},\\
\Hilb_{\tilde{E}_7}(t) &= \frac{[6][8][10][12][14][18][24][27][32][34][48][49][52][66][75]}{[3][4][5][7][9][11][13][17]}.
\end{align*}
\end{conj}
As for the remaining affine Dynkin diagram, $\FK_{\tilde{E}_8}$ appears finite-dimensional but too large for us to confidently infer a possible Hilbert series.

Even more basic is the following question.
\begin{question}
For which graphs $G$ is $\FK_G$ finite-dimensional?
\end{question}
While $\FK_G$ is finite-dimensional for all graphs $G$ on at most five vertices, it appears that some graphs on six vertices may have $\FK_G$ infinite-dimensional. While we are not yet able to prove that any $\FK_G$ is infinite-dimensional, our computations do suggest the following conjecture.
\begin{conj} \label{conj-infinite}
Let $G$ be a graph on six vertices. Then $\FK_G$ is infinite-dimensional if and only if it contains a subgraph isomorphic to one of the graphs shown in Figure~\ref{fig-infinite}.
\end{conj}
\begin{figure}
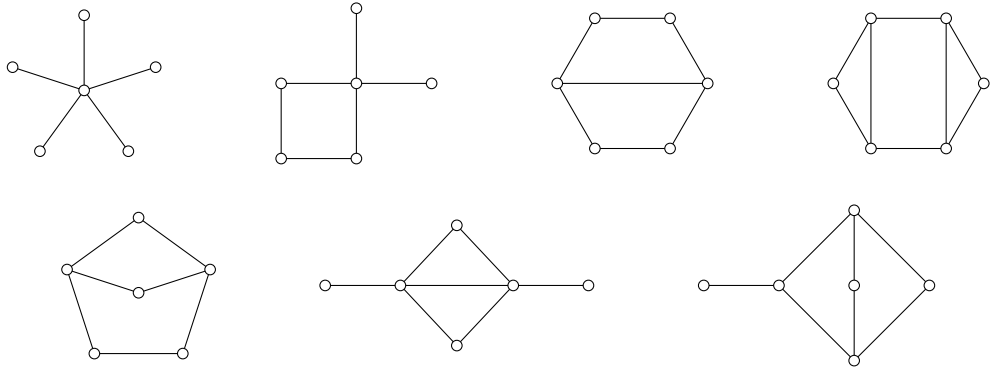

	\begin{center}
		\dr{(0,0)node[v]{}--(90:1)node[v]{} (0,0)--(18:1)node[v]{} (0,0)--(162:1)node[v]{} (0,0)--(234:1)node[v]{} (0,0)--(306:1)node[v]{}} \quad\quad\quad
		\dr{(1,0)node[v]{}--(0,0)node[v]{}--(-1,0)node[v]{}--(-1,-1)node[v]{}--(0,-1)node[v]{}--(0,0)--(0,1)node[v]{}}
		\quad\quad\quad
		\dr{(1,0)node[v]{}--(60:1)node[v]{}--(120:1)node[v]{}--(180:1)node[v]{}--(240:1)node[v]{}--(300:1)node[v]{}--(1,0)--(-1,0)}
		\quad\quad\quad
		\dr{(1,0)node[v]{}--(60:1)node[v]{}--(120:1)node[v]{}--(180:1)node[v]{}--(240:1)node[v]{}--(300:1)node[v]{}--(1,0) (60:1)--(-60:1) (120:1)--(-120:1)}\\[.5cm]
		\dr{(18:1)node[v]{}--(90:1)node[v]{}--(162:1)node[v]{}--(234:1)node[v]{}--(306:1)node[v]{}--(18:1)--(0,0)node[v]{}--(162:1)}\quad\quad\quad
		\dr{(0,0)node[v]{}--(1,0)node[v]{}--(2.5,0)node[v]{}--(1.75,.8)node[v]{}--(1,0)--(1.75,-.8)node[v]{}--(2.5,0)--(3.5,0)node[v]{}}\quad\quad\quad
		\dr{(-1,0)node[v]{}--(0,0)node[v]{}--(1,1)node[v]{}--(2,0)node[v]{}--(1,-1)node[v]{}--(0,0) (1,1)--(1,0)node[v]{}--(1,-1)}
	\end{center}
	\caption{\label{fig-infinite}Minimal graphs $G$ on six vertices for which $\FK_G$ appears to be infinite-dimensional. See Conjecture~\ref{conj-infinite}.}
\end{figure}

In particular, this would imply that $\FK_6$ is infinite-dimensional.

When $\FK_G$ is finite-dimensional, it appears that its Hilbert series is especially nice.
\begin{conj}
If $\FK_G$ is finite-dimensional, then $\Hilb_G(t)$ is a product of cyclotomic polynomials.
\end{conj}
This holds for all Hilbert series we have been able to compute, and it also appears plausible for those Hilbert series that we cannot compute in full. Note that the corresponding statement is also true for all finite Coxeter groups.

In Section~\ref{sec-coxeter}, we discussed many similarities between Coxeter groups and Fomin-Kirillov algebras.
\begin{question}
What other results about Coxeter groups have analogues for Fomin-Kirillov algebras? Are there corresponding notions to reflections/root systems/etc.?
\end{question}

We have studied the cases when $G$ is a Dynkin diagram of finite type because $\FK_G$ appears to be relatively simple in these cases. However, despite results such as Theorem~\ref{thm-weyl}, we know of no direct explanation for the apparent connection between $\FK_G$ and the corresponding Weyl group in these cases.
\begin{question}
Is there a uniform explanation for the structure of  $\FK_G$ when $G$ is a (simply-laced) Dynkin diagram of finite type? What if $G$ is an affine Dynkin diagram?
\end{question}

Our next question is related to the discussion in Section~\ref{sec-complementary}. By Corollary~\ref{cor-complementary}, there is a large class of graphs for which $\FK_{G_1} \otimes \FK_{G_2} \cong \FK_n$, but by Proposition~\ref{prop-counter}, this does not hold for all pairs of complementary graphs.
\begin{question}
For which complementary graphs $G_1$ and $G_2$ is it true that $\FK_{G_1} \otimes \FK_{G_2} \cong \FK_n$? In general, is the multiplication map $\FK_{G_1} \otimes \FK_{G_2} \to \FK_n$ always injective? What is the structure of $\FK_n$ as an $\FK_{G_1}$-$\FK_{G_2}$-bimodule?
\end{question}

Our next set of questions concerns relations in $\FK_G$.
\begin{question}
Is there a straightforward description of the relations of $\FK_G$?
\end{question}
In all of the examples we have discussed above, the minimal relations have a very special form: their $\S_n$-degrees are always automorphisms of the support of the relation (the graph containing the edges that appear in the relation).
\begin{question}
Must the $\S_n$-degree of a minimal relation be an automorphism of the support of the relation?
\end{question}
Also for many of the examples considered above, the minimal coset representatives of $\FK_H$ inside $\FK_G$ are relatively simple: the choice of minimal coset representatives is unique up to scalar factors (if we stipulate that they be representable by monomials). This allows us to describe a poset structure on the minimal coset representatives using the analogue of weak order. Then Theorem~\ref{thm-subgraph} suggests the following question.
\begin{question}
For which graphs $H \subset G$ does $\FK_G/\FK_H^+\FK_G$ have a unique monomial basis (up to scalar factors)?
\end{question}

Finally, recall Conjecture~\ref{conj-nichols}.
\begin{conj-nichols} \cite{MS}
The bilinear form $\langle\cdot,\cdot\rangle$ is nondegenerate on $\FK_n$.
\end{conj-nichols}
The quotient of $\FK_n$ by the kernel of this bilinear form is a certain type of braided Hopf algebra called a \emph{Nichols algebra} (see \cite{AS}). Nichols algebras exist in much greater generality than the examples just mentioned. For instance, the analogues of Fomin-Kirillov algebras for types other than $A$ as defined in \cite{KirillovMaeno} can be described in terms of Nichols algebras \cite{Bazlov}.  As such, it would be interesting to investigate the following question.
\begin{question}
To what extent can one extend the results of this paper to Fomin-Kirillov algebras of other types or more general Nichols algebras?
\end{question}

\section{Acknowledgments}

The authors would like to thank John Stembridge for his valuable input throughout the duration of this research, as well as Thomas Lam and Sergey Fomin for interesting discussions.

\section{Appendix}

Here we give the Hilbert series $\Hilb_G(t)$ for any connected graph on at most five vertices along with its top degree and total dimension. Note that $[k] = 1+t+t^2+\dots+t^{k-1}$.
\bigskip

\tikzset{every picture/.append style={scale=0.6}}
\begin{longtable}{c>{\qquad$}c<{$\qquad}cc}
$G$&\mbox{Hilb}&deg&dim\\
\hline\hline

\dr{(0,0) node[v]{} -- (1,0)node[v]{}}
&[2]&1&2\\

\hline

\dr{(0,0) node[v]{} -- (1,0)node[v]{} -- (2,0)node[v]{}}
&[2][3]&3&6\\

\dr{(0,0) node[v]{} -- (1,0)node[v]{} -- +(120:1)node[v]{} -- cycle}
&[2]^2[3]&4&12\\

\hline

\dr{(0,0) node[v]{} -- (1,0) node[v]{} -- (2,0) node[v]{} -- (3,0) node[v]{}}
&[2][3][4]&6&24\\

\dr{(0,0) node[v]{} -- +(150:1) node[v]{}  +(-150:1) node[v]{}--(0,0)--(1,0) node[v]{}}
&[3][4]^2&8&48\\

\dr{(0,0) node[v]{} -- +(150:1) node[v]{} -- +(-150:1) node[v]{}--(0,0)--(1,0) node[v]{}}
&[2][3][4]^2&9&96\\

\dr{(0,0)node[v]{}--(1,0)node[v]{}--(1,1)node[v]{}--(0,1)node[v]{}--cycle}
&[3]^2[4]^2&10&144\\

\dr{(0,0)node[v]{}--(1,0)node[v]{}--(1,1)node[v]{}--(0,1)node[v]{}--cycle--(1,1)}
&[2][3]^2[4]^2&11&288\\

\dr{(0,0)node[v]{}--(1,0)node[v]{}--(1,1)node[v]{}--(0,1)node[v]{}--cycle--(1,1) (1,0)--(0,1)}
&[2]^2[3]^2[4]^2&12&576\\

\hline

\dr{(0,0)node[v]{}--(1,0)node[v]{}--(2,0)node[v]{}--(3,0)node[v]{}--(4,0)node[v]{}}
&[2][3][4][5]&10&120\\

\dr{(0,0) node[v]{} -- +(150:1) node[v]{}  +(-150:1) node[v]{}--(0,0)--(1,0) node[v]{}--(2,0)node[v]{}}
&[4]^2[5][6]&15&480\\

\dr{(0,0)node[v]{} -- ++(45:1)node[v]{} -- +(45:1)node[v]{} +(135:1)node[v]{} -- +(0,0) -- +(-45:1)node[v]{}}
&[2]^{-2}[3]^{-2}[4]^2[5]^2[6]^4&28&14400\\
 
\dr{(0,0)node[v]{} -- ++(150:1)node[v]{} -- ++(-90:1)node[v]{} -- (0,0) -- (1,0)node[v]{} -- (2,0)node[v]{}}
&[2][4]^2[5][6]&16&960\\

\dr{(0,0)node[v]{}--(-1,0)node[v]{}--+(-150:1)node[v]{}--(-1,-1)node[v]{}--(-1,0)(-1,-1)--(0,-1)node[v]{}}&[4]^2[5][6]^2&20&2880\\

\dr{(0,0)node[v]{}--++(36:1)node[v]{}--++(-36:1)node[v]{}--++(-108:1)node[v]{}--++(-1,0)node[v]{}--cycle}&[4]^2[5][6]^2&20&2880\\

\dr{(0,0)node[v]{}--(0,1)node[v]{}--(-1,1)node[v]{}--(-1,0)node[v]{}--(0,0)--(1,0)node[v]{}}&[2]^{-1}[4]^2[5][6]^3&24&8640\\

\dr{(0,0)node[v]{} -- ++(-150:1)node[v]{} -- ++(-150:1)node[v]{} -- ++(90:1)node[v]{} -- ++(-30:1) -- ++(-30:1)node[v]{}}
&[2]^{-1}[3]^{-2}[4]^2[5]^2[6]^4&29&28800\\

\dr{(2,0)node[v]{}--(1,0)node[v]{}--(0,0)node[v]{}--(0,1)node[v]{}--(1,1)node[v]{}--(1,0)node[v]{} (0,0)--(1,1)}&[4]^2[5][6]^3&25&17280\\ 

\dr{(0,0)node[v]{}--++(150:1)node[v]{}--++(0,-1)node[v]{}--++(30:1)--++(30:1)node[v]{}--++(0,-1)node[v]{}--cycle}&[3]^{-2}[4]^2[5]^2[6]^4&30&57600\\

\dr{(0,0)node[v]{}--(-1,0)node[v]{}--(-1,1)node[v]{}--(0,1)node[v]{}--(30:1)node[v]{}--(0,0)--(0,1)node[v]{}}
&[2]^{-1}[3]^{-1}[4]^3[5][6]^4&30&69120\\

\dr{(2,0)node[v]{}--(1,0)node[v]{}--(0,0)node[v]{}--(0,1)node[v]{}--(1,1)node[v]{}--(1,0)node[v]{} (1,0)--(0,1)}&[2]^{-1}[3]^{-1}[4]^2[5]^2[6]^4&31&86400\\

\dr{(0,0)node[v]{}--(1,0)node[v]{}--(1,1)node[v]{}--(0,1)node[v]{}--cycle--(1,1) (.5,.5)node[v]{}}
&[2]^{-3}[3]^{-1}[4]^4[5]^2[6]^4&35&345600\\

\dr{(1,0)node[v]{}--(0,0)node[v]{}--(-1,0)node[v]{}--(-1,1)node[v]{}--(0,1)node[v]{}--(0,0)--(-1,1) (-1,0)--(0,1)}
&[3]^{-1}[4]^2[5]^2[6]^4&32&172800\\

\dr{(120:1)node[v]{}--(-1,0)node[v]{}--(0,0)node[v]{}--(120:1)--(60:1)node[v]{}--(1,0)node[v]{}--(0,0)--(60:1)}
&[2]^{-1}[3]^{-1}[4]^3[5]^2[6]^4&34&345600\\

\dr{(0,0)node[v]{}--(1,0)node[v]{}--(1,1)node[v]{}--(0,1)node[v]{}--cycle--(1,1) (.5,.5)node[v]{}--(1,0)}
&[2]^{-2}[3]^{-1}[4]^4[5]^2[6]^4&36&691200\\

\dr{(0,0)node[v]{}--(-1,0)node[v]{}--(0,1)node[v]{}--(0,0)--(-1,1)node[v]{}--(0,1)--(30:1)node[v]{}--(0,0)}
&[2]^{-2}[3]^{-1}[4]^4[5]^2[6]^4&36&691200\\

\dr{(-1,0)node[v]{}--(-1,1)node[v]{}--(0,1)node[v]{}--(30:1)node[v]{}--(0,0)node[v]{}--(-1,0)--(0,1)--(0,0)--(-1,1)}
&[2]^{-1}[3]^{-1}[4]^4[5]^2[6]^4&37&1382400\\

\dr{(0,0)node[v]{}--(1,0)node[v]{}--(1,1)node[v]{}--(0,1)node[v]{}--(0,0)--(1,1) (1,0)--(0,1) (.5,.5)node[v]{}}
&[2]^{-2}[4]^4[5]^2[6]^4&38&2073600\\

\dr{(306:1)node[v]{}--(18:1)node[v]{}--(90:1)node[v]{}--(162:1)node[v]{}--(234:1)node[v]{}
--(18:1)--(162:1)--(306:1)--(90:1)--(234:1)}
&[2]^{-1}[4]^4[5]^2[6]^4&39&4147200\\

\dr{(306:1)node[v]{}--(18:1)node[v]{}--(90:1)node[v]{}--(162:1)node[v]{}--(234:1)node[v]{}
--(18:1)--(162:1)--(306:1)--(90:1)--(234:1)--(306:1)}
&[4]^4[5]^2[6]^4&40&8294400

\end{longtable}

\bibliography{fk}
\bibliographystyle{plain}

\end{document}